\documentclass[final]{siamltex}

\usepackage{showkeys}

\newcommand{\eps}{\ensuremath{\epsilon}}

\newcommand{\aA}{\ensuremath{\mathcal{A}}}
\newcommand{\bB}{\ensuremath{\mathcal{B}}}

\newcommand{\dD}{\ensuremath{\mathcal{D}}}

\newcommand{\fF}{\ensuremath{\mathcal{F}}}

\newcommand{\pP}{\ensuremath{\mathcal{P}}}

\renewcommand{\phi}{\varphi}

\title{Random attractors for stochastic evolution equations driven by fractional Brownian motion
\thanks{H. Gao: Supported by a China NSF Grant
11171158, National Basic Research Program of China (973 Program) No. 2013CB834100,  the Natural Science Foundation of Jiangsu Province (BK2011777),
Qing Lan and "333" Project of Jiangsu Province and the NSF of the Jiangsu Higher Education
Committee of China (11KJA110001). \newline M.J. Garrido-Atienza and B. Schmalfu{\ss}: Partially supported by the European Funds for Regional Development and Ministerio de Economia y Competitividad (Spain) under grant MTM2011-22411.}}

\author{Hongjun Gao\thanks{Institute of Mathematics,
School of Mathematical Sciences, Nanjing Normal University, Nanjing 210046, China, ({\tt gaohj@njnu.edu.cn}).}
\and
Mar\'{\i}a J. Garrido-Atienza\thanks{Dpto. Ecuaciones Diferenciales y An\'alisis Num\'erico, Universidad de Sevilla, Apdo. de Correos 1160, 41080-Sevilla, Spain, ({\tt mgarrido@us.es}).}
\and
Bj{\"o}rn Schmalfu{\ss }\thanks{Institut f\"{u}r Mathematik,
Institut f{\"u}r Stochastik, Ernst Abbe Platz 2, 07737, Jena, Germany,  ({\tt bjoern.schmalfuss@uni-jena.de}).}}

\begin{document}

\maketitle

\begin{abstract}
The main goal of this article is to prove the existence of a random attractor for a stochastic evolution equation driven by a fractional Brownian motion with $H\in (1/2,1)$. We would like to emphasize that we do not use the usual cohomology method, consisting of transforming the stochastic equation into a random one, but we deal directly with the stochastic equation. In particular, in order to get adequate a priori estimates of the solution needed for the existence of an absorbing ball, we will introduce stopping times to control the size of the noise.
In a first part of this article we shall obtain the existence of a pullback attractor for the non-autonomous dynamical system generated by the pathwise mild solution of an nonlinear infinite-dimensional evolution equation with non--trivial H\"older continuous driving function. In a second part, we shall consider the random setup: stochastic equations having as driving process a fractional Brownian motion with $H\in (1/2,1)$. Under a smallness condition for that noise we will show the existence and uniqueness of a random attractor for the stochastic evolution equation.
\end{abstract}

\begin{keywords}
Fractional derivatives; pathwise mild solutions; non-autonomous and random dynamical systems; fractional Brownian motion; pullback and random attractor.
\end{keywords}

\begin{AMS} 37L55, 60H15, 37L25, 35R60, 58B99.
\end{AMS}

\pagestyle{myheadings}
\thispagestyle{plain}
\markboth{H. Gao, M.J. Garrido-Atienza and B. Schmalfu\ss}{Random Attractors for SPDEs driven by an fBm}

\section{Introduction}
This article shows the existence of random attractors for a new class of stochastic evolution equations.
These equations contain a nontrivial fractional noise with a Hurst parameter $H>1/2$.
In particular, it generalizes  the
results of the conference proceedings by Garrido-Atienza {\it et al.} \cite{GMS08}, where random attractors are studied for ordinary
stochastic equations containing a fractional noise. The idea of this article is based on the modern theory of stochastic integration
for fractional Brownian motions (fBm) with $H>1/2$, see for instance Z{\"a}hle \cite{Zah98}. These integrals are related to the Young integration \cite{You36} and one of its presentations is based on a sort of generalized integration by parts formula with respect to fractional derivatives. The main advantage of this integration with respect to the classical It\^o integration theory, where the integrator is a white noise (or equivalently a fractional Brownian motion with Hurst parameter $H=1/2$), is that integrals can be defined pathwise. However, It\^o integrals are only defined almost surely and exceptional sets may depend on the integrand. But this fact contradicts the definition of a random dynamical system, where initial state dependent exceptional sets are not allowed. By using fractional derivatives to define stochastic integrals we are able to avoid this dependence on exceptional sets. \\
We remark that the results presented in this article do not cover the white noise case. For the existence of pathwise solutions
for this case (and more general for the cases in which $H\in (1/3,1/2]$) we refer to Garrido-Atienza {\it et al.} \cite{GLS12a}, \cite{GLS12b} and the forthcoming paper \cite{Gar12}. We want to emphasize here that the techniques to obtain such a pathwise solution are qualitatively different from the methods presented in this article.\\
This article can be seen to be divided into two parts. In the first part, Sections 2 and 3, we mention that nonlinear infinite-dimensional evolution equations driven by a H\"older continuous function with H\"older index greater than $1/2$ have a pathwise mild solution which generates a non-autonomous dynamical system, and obtain the existence of a unique pullback attractor associated to it if we restrict this dynamical system to a discrete time set. Existence results for this kind of equations have been already studied in Maslowski and Nualart \cite{MasNua03} and in Garrido-Atienza {\it et al.} \cite{GLS09} when the driving function is an fBm with $H>1/2$, and very recently by Chen {\it et al.} \cite{ChGGSch12}. In this last paper, the authors have been able to overcome the lack of the regularity of the semigroup generated by the linear part of the evolution equation, by considering suitable modifications of the space of $\beta$--H{\"o}lder continuous functions as phase space, where $\beta>1/2$ is related to the H\"older index of the driving function, and have established the existence and uniqueness of a pathwise mild solution. These results are thus shortly included as a preliminary step, where the construction of the integral for H\"older-continuous integrators is the main tool. Then, by standard arguments we can consider the pathwise mild solution to be a non-autonomous dynamical system. According to Flandoli and Schmalfu\ss \, \cite{FlaSchm96} or Schmalfu\ss \, \cite{Schm92}, in order to prove the existence of a non-autonomous or pullback attractor, appropriate a priori estimates of the solution are necessary. However,
by the particular structure of the estimates of the pathwise integrals, the standard Gronwall lemma is not available and, therefore, we have to formulate a special Gronwall lemma for discrete time, which will be shown to work only when the estimates of these integrals are not too large. To ensure that this smallness condition holds true we will introduce (non-Markovian) stopping times with the property that they themselves satisfy the cocycle property, and assuming that these stopping times have special asymptotic properties, the existence of a unique pullback attractor is obtained.

Then, in a second part of the article (Section 4), we show how to adapt the previous results to obtain the random attractor associated to stochastic evolution equations driven by an fBm with $H>1/2$. In a first step, by choosing the canonical version of the fBm, we show that the non-autonomous dynamical system becomes a random dynamical system and that the stopping times are measurable. In addition, it suffices to choose fractional noises with covariance operators having small trace, since this condition ensures that the required asymptotical properties of the stopping times hold with probability one.
All the previous study finally concludes with the existence and uniqueness of a random attractor.

\section{Preliminaries} In this section we collect definitions and properties of dynamical systems and present the construction of the integral with H\"older continuous integrator and H\"older exponent greater than $1/2$. In spite of the fact that this construction has been already done in the recent paper \cite{ChGGSch12}, we present it here for the sake of completeness, since the well understanding of this integral is the starting point to obtain the cocycle property and afterwards the random attractor associated to the corresponding stochastic evolution equation.

\subsection{Dynamical systems}\label{s2}
Let $(V,|\cdot|)$ be a Banach space and let $T^+=R^+$ or $Z^+$. A mapping $\phi:T^+\times V\to V$
having the {\em semigroup} property
\[\phi(t,\cdot)\circ\phi(\tau,u_0)=\phi(t+\tau,u_0),\qquad\phi(0,u_0)=u_0\qquad t,\,\tau\in T^+, \quad u_0\in V\]
is called an {\em (autonomous) dynamical system}. A dynamical system $\phi$ has a {\em global attractor} $\aA\subset V$ with respect to the set system $\dD$, consisting of the bounded sets of $V$, if $\aA$ is non-empty, compact, invariant, that is
$$\varphi(t,\aA)=\aA,  \qquad {\rm for \, all }\ t \in T^+$$
 and
\[\lim_{T^+\ni t\to\infty}{\rm dist}(\phi(t,D),\aA)=0 \qquad {\rm for \, every }\ D\in\dD.\]
${\rm dist}(X,Y)$ denotes the Hausdorff-semidistance defined by
\begin{eqnarray*}
    {\rm dist}(X,Y)=\sup_{x\in X}\inf_{y\in Y}|x-y|.
\end{eqnarray*}
A comprehensive presentation of the concept of attractors can be found in the monographs by Babin and Vishik \cite{BaVis92}, Hale
\cite{Hale} or Temam \cite{Tem97}.

\medskip

We want to consider a generalization of the concept of global attractors to {\em non-autonomous} and {\em random dynamical systems}  given by the so called {\em pullback attractors}. As a first ingredient, for the {\em time set} $T=R$ or $Z$, we introduce the {\em flow} $(\theta_t)_{t\in T}$ on the set $\Omega$ of {\em non-autonomous perturbations} by
\begin{eqnarray*}
&&\theta:T\times \Omega\to\Omega\\
    &&\theta_t\circ\theta_\tau=\theta_{t+\tau},\quad\theta_0\omega=\omega\quad {\rm for }\ t,\tau\in T,\;\omega\in\Omega.
\end{eqnarray*}
The easiest example for such a flow is given by $\Omega=T$ and $\theta_ji=i+j$ for $j\in T$.

As a generalization of the semigroup property we consider a {\em cocycle} to be a mapping $\phi:T^+\times\Omega\times V\to V$ such that
\begin{eqnarray}
   \phi(t+\tau,\omega,u_0)&=&\phi(t,\theta_\tau\omega,\cdot)\circ\phi(\tau,\omega,u_0),\nonumber \\[-1.5ex]
    \label{cocycle}\\[-1.5ex]
\phi(0,\omega,u_0)&=&u_0,\nonumber
\end{eqnarray}
for all $t,\,\tau\in T^+$, $u_0\in V$ and $\omega\in\Omega$. $\phi$ is also called a {\em non-autonomous dynamical system.}\\

Let us now equip $(\Omega,\theta)$ with a measurable structure. Consider the probability space $(\Omega,\fF,P)$ where $\fF$ is a $\sigma$-algebra on $\Omega$ and $P$ is a measure invariant and {\em ergodic} with respect to $\theta$.
Then $(\Omega,\fF,P,\theta)$ is called an {\em ergodic metric dynamical system}. A $\bB(T^+)\otimes\fF\otimes \bB(V),\bB(V)$
measurable mapping $\phi$ having the cocycle property (\ref{cocycle}) is called a {\em random dynamical system (RDS)} (with respect to the
metric dynamical $(\Omega,\fF,P,\theta)$).

\medskip

Before we give the notion of an attractor for a non-autonomous dynamical system we introduce set systems that will be attracted
by that attractor. Let $\dD$ be a set of families $(D(\omega))_{\omega\in\Omega}$ such that $\emptyset\not=D(\omega)\subset V$ and satisfying a general property $\pP$, which will be determined later. In addition we assume the following completeness condition for $\dD$: let $D^\prime=(D^\prime(\omega))_{\omega\in\Omega}$ such that
$\emptyset\not=D^\prime (\omega)\subset V$ and satisfies the property $\pP$, if moreover there exists a $D\in\dD$ such that $D^\prime(\omega)\subset D(\omega)$ for $\omega\in\Omega$, then $D^\prime\in \dD$. For a given $\nu>0$, an example of such a system is the {\em backward $\nu$--exponentially growing sets} $\dD^\nu$: $D=(D(\omega))_{\omega\in\Omega}$ is an element of this set system if there exists a mapping $\Omega\ni \omega\mapsto r(\omega)\in R^+$ such that $\emptyset
\not=D(\omega)\subset B_V(0,r(\omega))$ and
\begin{eqnarray}\label{eq11}
    \limsup_{T^\ni t\to-\infty}\frac{\log^+ r(\theta_t\omega)}{|t|}<\nu.
\end{eqnarray}
We sometimes write $\dD_{Z,V}^\nu$ or $\dD_{R,V}^\nu$ to emphasize the corresponding time set.

We now consider these families of sets in a random set up. Assuming  that $V$ is a {\em separable} Banach space, we consider
$D=(D(\omega))_{\omega\in\Omega}$ such that $\emptyset\not=D(\omega)\subset V$ is closed, and verifying the following property $\pP$: the mapping
\begin{eqnarray*}
    \omega\mapsto {\rm dist}(\{u\},D(\omega))
\end{eqnarray*}
is a random variable for $u\in V$. Then $D$ is a random set, see Castaing and Valadier \cite{CasVal77} Chapter III. Indeed, we are interested in considering the set $\hat\dD$ of random sets having a backward and forward exponential growth ($\nu=0$), which means that for a random set $D$ there exists a positive random variable $r$ such that $D(\omega)\subset B_V(0,r(\omega))$ and
\begin{eqnarray}\label{eq12}
    \lim_{T\ni  t\to\pm\infty}\frac{\log^+ r(\theta_t\omega)}{|t|}=0\qquad {\rm for }\ \omega\in\Omega.
\end{eqnarray}
These sets are also called the {\em random tempered sets}. In this case we are only interested in the time set $T=R$.

By  the following lemma it is only interesting to consider the case $\nu=0$.\\
\begin{lemma}\label{l8}
Let $(\Omega,\fF,P)$ be an ergodic metric dynamical system. If $D$ is a $\nu$--exponentially growing random set ($\nu<\infty$) then $D$ is just exponentially growing.
\end{lemma}

The result follows directly, since by Arnold \cite{Arn98} Proposition 4.1.3 we have that
\begin{eqnarray}
  \lim_{t\to+\infty}\frac{\log^+ r(\theta_t\omega)}{t} &=& \lim_{t\to-\infty}\frac{\log^+ r(\theta_t\omega)}{|t|}\in\{0,+\infty\}.\label{eq212}
\end{eqnarray}
\\

A family $\aA=(\aA(\omega))_{\omega\in\Omega}$ is a {\em pullback attractor} for the non-autonomous dynamical system $\phi$ with respect to $\dD$ if
\begin{eqnarray*}
&&\aA\in\dD,\quad \aA(\omega)\not=\emptyset\quad{ \rm compact},\\
 &   &\phi(t,\omega,\aA(\omega))=\aA(\theta_t\omega),\quad {\rm for \, all }\ t\in T^+,\,\omega\in\Omega,\\
   & &\lim_{T\ni t\to\infty}{\rm dist}(\phi(t,\theta_{-t}\omega,D(\theta_{-t}\omega)),\aA(\omega))=0,\quad {\rm for \, all }\ \omega\in\Omega,\,D\in \dD.
\end{eqnarray*}
In the case that $\phi$ is an RDS and $\hat\dD$ consists of the random tempered sets introduced above we call $\aA$ a random attractor. This notion has been introduced in Schmalfu{\ss} \cite{Schm92}.

Finally, a family $B=(B(\omega))_{\omega\in\Omega}$ is called pullback absorbing for $\dD$ if
\begin{eqnarray*}
    \phi(t,\theta_{-t}\omega,D(\theta_{-t}\omega))\subset B(\omega) \quad {\rm for }\ t\ge \tilde T(D,\omega)\in T^+
\end{eqnarray*}
for any $D\in\dD$ and $\omega\in\Omega$. $\tilde T(D,\omega)$ is the so-called absorption time. \\

We have the following main theorem about the existence of a pullback/random attractor (see Flandoli and Schmalfu{\ss} \cite{FlaSchm96} or Schmalfu{\ss} \cite{Schm99a}):

\begin{theorem}\label{t3}
Let $\phi$ be a continuous non-autonomous dynamical system.  Suppose that $\phi$ has a compact pullback absorbing set $B$ in $\dD$.
Then $\phi$ has a unique pullback attractor with respect to $\dD$. If $\phi$ is a random dynamical system and $\dD=\hat \dD$ consists of the
tempered sets introduced above, then the associated pullback attractor is a random attractor. In both cases
$$\aA(\omega)=\bigcap_{\tau \geq T_0(B,\omega)} \overline{\bigcup_{t\geq \tau} \phi(t,\theta_{-t} \omega,B(\theta_{-t}\omega))}$$
being $T_0(B,\omega)$ the absorbing time corresponding to $B$.
\end{theorem}

\subsection{Integrals   in Hilbert-spaces for H{\"o}lder continuous integrators with H{\"o}lder exponent $>1/2$}\label{ss2.2}

We begin this subsection by introducing some function spaces. Assume $(V,(\cdot,\cdot),|\cdot|)$ is a separable Hilbert space with the complete orthonormal base $(e_i)_{i\in N}$. Let $C^\beta([T_1,T_2];V)$ be the Banach space of {\em H{\"o}lder continuous} functions with exponent $\beta>0$ having values in $V$. A norm on this space is given by
\begin{eqnarray*}
    \|u\|_{\beta}=\|u\|_{\beta,T_1,T_2}=\sup_{s\in [T_1,T_2]}|u(s)|+|||u|||_{\beta,T_1,T_2},
\end{eqnarray*}
with
\begin{eqnarray*}
|||u|||_{\beta,T_1,T_2}=\sup_{T_1\le s<t\le T_2}\frac{|u(t)-u(s)|}{|t-s|^\beta}.
\end{eqnarray*}
$C([T_1,T_2];V)$ denotes the space of  continuous functions on $[T_1,T_2]$ with values in $V$ with finite supremum norm, and let $C^{\beta,\sim}([T_1,T_2];V) \subset C([T_1,T_2];V)$ equipped with the norm
\begin{eqnarray*}
    \|u\|_{\beta,\sim}=\|u\|_{\beta,\sim,T_1,T_2}=\sup_{s\in[T_1,T_2]}|u(s)|+\sup_{T_1< s<t\le T_2}(s-T_1)^\beta\frac{|u(t)-u(s)|}{|t-s|^\beta}.
\end{eqnarray*}
For every $\rho>0$ we can consider the equivalent norm

\begin{eqnarray}
\|u\|_{\beta,\rho,\sim,T_1,T_2}&=&\sup_{s\in[T_1,T_2]}e^{-\rho(s-T_1)}|u(s)|\nonumber \\
&+&\sup_{T_1< s<t\le T_2}(s-T_1)^\beta e^{-\rho(t-T_1)}\frac{|u(t)-u(s)|}{|t-s|^\beta}.\label{eq38}
\end{eqnarray}
When no confusion is possible, we will write $\|u\|_{\beta,\rho,\sim}$ instead of $\|u\|_{\beta,\rho,\sim,T_1,T_2}$.\\

\begin{lemma}\label{l1} (see \cite{ChGGSch12})
$C^{\beta,\sim}([T_1,T_2];V)$  is a Banach space.\\
\end{lemma}

Now we introduce integrals where the integrator is a H{\"o}lder continuous function. The definition of these integrals shall have as basement the definition and properties of the pathwise integrals given by Z{\"a}hle \cite{Zah98}. Let us then assume that $\tilde V,\,\hat V$ are separable Hilbert spaces, then for $0<\alpha<1$ and general measurable functions $K:[T_1,T_2]\mapsto \hat V$ and $\omega:[T_1,T_2]\mapsto \tilde V$, we define their Weyl fractional derivatives by
\begin{eqnarray*}
    D_{{T_1}+}^\alpha K[r]=\frac{1}{\Gamma(1-\alpha)}\bigg(\frac{K(r)}{(r-T_1)^\alpha}+\alpha\int_{T_1}^r\frac{K(r)-K(q)}{(r-q)^{1+\alpha}}dq\bigg)\in \hat V,\,\\
    D_{{T_2}-}^{1-\alpha} \omega_{T_2-}[r]=\frac{(-1)^{1-\alpha}}{\Gamma(\alpha)}
    \bigg(\frac{\omega(r)-\omega(T_2-)}{(T_2-r)^{1-\alpha}}
    +(1-\alpha)\int_r^{T_2}\frac{\omega(r)-\omega(q)}{(q-r)^{2-\alpha}}dq\bigg)\in
    \tilde V,
\end{eqnarray*}
where $ \omega_{T_2-}(r)= \omega(r)- \omega(T_2-)$, being $\omega(T_2-)$ the left side limit of $\omega$ at $T_2$. \\

Let $f\in L^1((T_1,T_2); R)$ and $\alpha>0$. Following Samko {\it et al.} \cite{Samko}, the left-sided and right-sided fractional Riemann-Liouville integrals of $f$ of order $\alpha$ are defined for almost all $x\in (T_1,T_2)$ by
\begin{eqnarray*}
I_{T_1+}^\alpha f(x) & = \frac{1}{\Gamma(\alpha)} \int_{T_1}^x (x-r)^{\alpha-1} f(r) dr,\\
I_{T_2-}^{1-\alpha } f(x)&= \frac{(-1)^{1-\alpha}}{\Gamma (\alpha)} \int_x^{T_2} (r-x)^{\alpha-1} f(r) dr
\end{eqnarray*}
respectively, where $\Gamma (\alpha)$ is the Euler function. Denote by $I_{T_1+}^\alpha (L_p((T_1,T_2);R))$ (resp. $I_{T_2-}^{\alpha} (L_{p^\prime}((T_1,T_2); R))$) the image of $L_p((T_1,T_2);R)$ (resp. $L_{p^\prime}((T_1,T_2); R))$) by the operator $I_{T_1+}^\alpha $ (resp. $I_{T_2-}^{\alpha}$).

Suppose now that $k \in I_{T_1+}^\alpha (L_p((T_1,T_2);R)),\, \zeta_{T_2-} \in
I_{T_2-}^{\alpha} (L_{p^\prime}((T_1,T_2); R))$ with $1/p+1/{p^\prime}\le 1$, being $\alpha p<1$, and that $k(T_1+)$, the right side limit of $k$ at $T_1$, exists. Then following Z\"ahle \cite{Zah98} we define
\begin{eqnarray}\label{eq10bis}
    \int_{T_1}^{T_2} kd\zeta=(-1)^\alpha\int_{T_1}^{T_2} D_{T_1+}^\alpha k[r]D_{T_2-}^{1-\alpha}\zeta_{T_2-}[r]dr,
\end{eqnarray}
where $ \zeta_{T_2-}(r)= \zeta(r)- \zeta(T_2-)$, for $r\in (T_1,T_2)$. In fact, under the above assumptions, the appearing fractional derivatives in the integrals are well defined taking $\tilde V=\hat V=R$.\\

Consider now the separable Hilbert space $L_2(V)$ of Hilbert-Schmidt operators on $V$ with the usual norm $\|\cdot\|_{L_2(V)}$ and inner product $(\cdot,\cdot)_{L_2(V)}$. A base in this space is given by
\begin{eqnarray}\label{base}    E_{ij}e_k=\left\{\begin{array}{lcl}
    0&:& j\not= k\\
    e_i &:& j= k.
    \end{array}
    \right.
\end{eqnarray}
Let us consider mappings $K:[0,T]\to L_2(V)$ and $\omega:[0,T]\to V$, such that $k_{ji}=(K,E_{ji})_{L_2(V)}\in I_{T_1+}^\alpha (L_p((T_1,T_2);R))$ and $k_{ji}(T_1+)$ exists and $\alpha p<1$. Moreover,  $\zeta_{iT_2-}=(\omega_{T_2-}(t),e_i)\in I_{T_2-}^{1-\alpha} (L_{p^\prime}((T_1,T_2); R))$, being $1/p+1/p^\prime\le 1$. In addition,
\begin{eqnarray*}
   [T_1,T_2]\ni r\mapsto  \|D_{T_1+}^\alpha K[r]\|_{L_2(V)}|D_{T_2-}^{1-\alpha} \omega_{T_2-}[r]|\in L_{1}((0,T);R).
\end{eqnarray*}
We then introduce
\begin{eqnarray}\label{eq3}
    \int_{T_1}^{T_2} K d\omega:= (-1)^\alpha\int_{T_1}^{T_2} D_{T_1+}^\alpha K[r]D_{T_2-}^{1-\alpha}\omega_{T_2-}[r]dr.
\end{eqnarray}
Due to Pettis' theorem and the separability of $V$ the integrand is weakly measurable and hence measurable. In addition, we can present this integral by
\begin{eqnarray*}
    \int_{T_1}^{T_2}Kd\omega=\sum_{j}\bigg(\sum_i\int_{T_1}^{T_2}
    D_{T_1+}^{\alpha}k_{ji}[r]D_{T_2-}^{1-\alpha}\zeta_{iT_2-}[r]dr \bigg) e_j,
\end{eqnarray*}
with norm given by
\begin{eqnarray*}
    \bigg|\int_{T_1}^{T_2}Kd\omega\bigg|
    &=&\bigg(\sum_j\bigg|\sum_{i}\int_{T_1}^{T_2}D_{T_1+}^{\alpha}k_{ji}[r]D_{T_2-}^{1-\alpha}\zeta_{iT_2-}[r]dr\bigg|^2\bigg )^\frac12\\
    &\le& \int_{T_1}^{T_2}\|D_{T_1+}^\alpha K[r]\|_{L_2(V)}|D_{T_2-}^{1-\alpha} \omega_{T_2-}[r]| dr.
\end{eqnarray*}

For $H>1/2$, in what follows we fix parameters $1/2<\beta<\beta^\prime<\beta^{\prime\prime}<H$. Under this choice, let $\Omega$ be the $(\theta_t)_{t\in R}$-invariant set of {\em paths} $\omega:R\to V$ which are $\beta^{\prime\prime}$-H{\"o}lder continuous on any compact subinterval of $R$ and are zero at zero.
Later we will need to formulate asymptotical conditions for the set of these paths.\\
As the flow $(\theta_t)_{t\in R}$ on $\Omega$ of non-autonomous perturbations we consider
\begin{eqnarray}\label{shift}
\theta:R\times \Omega\to\Omega, \qquad \theta_t \omega(\cdot)=\omega(t+\cdot)-\omega(t).
\end{eqnarray}

\begin{lemma}(see \cite{ChGGSch12}) \label{l3}
Suppose that $K:[T_1,T_2]\mapsto L_2(V)$ is $\beta$-H{\"o}lder--continuous where $1-\beta^\prime<\alpha<\beta$. Then the  integral
$$
\int_{T_1}^{T_2} K(r) d\omega(r)\in V
$$
is well defined in the sense of (\ref{eq3}).  If for $0\le \tau\le T_1$ the mapping $K:[T_1-\tau,T_2]\mapsto L_2(V)$ is $\beta$-H{\"o}lder--continuous, then
$$
\int_{T_1}^{T_2} K(r) d\omega(r)=\int_{T_1-\tau}^{T_2-\tau} K(r+\tau) d\theta_\tau\omega(r).
$$

\end{lemma}

\section{Evolution equations driven by an integral with H{\"o}lder continuous integrator}

We now consider the following evolution equation on $[0,T]$:
\begin{eqnarray}\label{eq6}
    du=Au\,dt+F(u)\,dt+G(u)\,d\omega,\qquad u(0)=u_0\in V
\end{eqnarray}
driven by a H{\"o}lder continuous path $\omega$. The integral with respect to $d\omega$ is interpreted in the sense of the previous section.

In what follows we describe the different terms on the right hand side of (\ref{eq6}). Let $-A$ be a strictly positive and symmetric operator
with a compact inverse which  is the {\em generator} of an {\em analytic exponential decreasing  semigroup} $S$ on $V$. We introduce the spaces $V_\delta:=D((-A)^\delta)$ with norm $|\cdot|_{V_\delta}$ for $\delta\ge 0$ such that $V=V_0$. The spaces $V_\delta,\,\delta>0$ are supposed to be compactly embedded in $V$. From now on, assume that $(e_i)_{i\in N}$ is the complete orthonormal base generated by the eigenelements of $-A$ with associated eigenvalues $(\lambda_i)_{i\in N}$.

Let $L(V_\delta,V_\gamma)$ denote the space of continuous linear operators from $V_\delta$ into $V_\gamma$.
Then there exists a constant $c>0$ such that we have the estimates
\begin{eqnarray}
  |S(t)|_{L(V_\alpha, V_{\gamma})}=|(-A)^\gamma S(t)|_{L(V_\alpha,V)}\le
  \frac{c}{t^{\gamma-\alpha}}e^{-\lambda_1 t}\quad {\rm for }\
  \gamma\geq \alpha,\quad t\in (0,T]\label{eq1}.
  \end{eqnarray}
  \begin{eqnarray}
 |S(t)-{\rm id}|_{L(V_{\sigma+\mu},V_{\theta+\mu})} &\le c
t^{\sigma-\theta}, \quad {\rm for }\ \theta\geq 0,\quad \sigma\in
(\theta,1+\theta],\quad \mu\in R \label{eq2},
\end{eqnarray}
see \cite{chueshov}, Chapter 2. In particular, formula (1.14) on page 83 of that book becomes (\ref{eq1}) above in bounded intervals, where we shall use it.


From these inequalities, for $\mu,\,\eta\in (0,1]$ and $0\le \delta<\gamma+\mu$ there exists a $c>0$ such that for $0\leq q\leq r\leq s\leq t$, we can derive that
\begin{eqnarray*}
& & \qquad |S(t-r)-S(t-q)|_{L(V_{\delta},V_{\gamma})}\le c(r-q)^\mu(t-r)^{-\mu-\gamma+\delta}, \\[-1.5ex]
\\[-1.5ex]
& & \qquad |S(t-r)- S(s-r)-S(t-q)+S(s-q)|_{L(V)}\leq
c(t-s)^{\mu}(r-q)^{\eta}(s-r)^{-(\mu+\eta)}.\nonumber
\end{eqnarray*}

Furthermore, the mapping $F:V\to V$ is supposed to be Lipschitz continuous, but we are going to assume that $F\equiv 0$, simplification that we make for the sake of brevity since the $dt$-nonlinearity is not the interesting term in the problem (\ref{eq6}), and of course that we would achieve the same existence results as we obtain below assuming that $F$ were Lipschitz. However, for the existence of an attractor we also would need to assume for $F$ to have a sufficiently small Lipschitz constant.

The non-linear operator $G: V\to L_2(V)$ is assumed to be a twice continuously Fr\'echet--differentiable operator with bounded first and second
derivatives. Let us denote, respectively, by $c_{DG},\, c_{D^2G}$ the bounds for these derivatives and set $c_G=\|G(0)\|_{L^2(V)}$.
Then, for $u_1,\,u_2,\,v_1,\,v_2\in V$, we have
\begin{eqnarray*}
  &&  \|G(u_1)-G(v_1)\|_{L_2(V)}\le c_{DG}|u_1-v_1|,\\
   &&  \|G(u_1)\|_{L_2(V)}\le c_G+c_{DG}|u_1|,\\
  &&  \|G(u_1)-G(v_1)-(G(u_2)-G(v_2))\|_{L_2(V)}\\
   && \quad \le c_{DG}|u_1-v_1-(u_2-v_2)|+c_{D^2G} |u_1-u_2|(|u_1-v_1|+|u_2-v_2|)
\end{eqnarray*}
(the proof of the last inequality can be found in \cite{MasNua03}).

We are interested in solving the equation (\ref{eq6}) in a mild sense, which means that for $t\in [0,T]$ we have to solve
\begin{eqnarray}\label{eq7}
    u(t)=S(t)u_0+\int_0^tS(t-r)G(u(r))d\omega.
\end{eqnarray}

According to the definition of the stochastic integral given in the previous section, we need to estimate the fractional derivative of the term $S(t-\cdot)G(u(\cdot))$.
\begin{lemma}\label{l18} (see \cite{ChGGSch12})
Let $S$ be the semigroup generated by $A$, assume that $G$ is Lipschitz and $u\in C^{\beta,\sim}([T_1,T_2];V)$.
Then for the orthonormal base $(e_i)_{i\in N}$ of $V$, for any $i,\,j\in N$,
\begin{eqnarray*}
     (e_i,S(t-\cdot)G(u(\cdot))e_j) \in I_{T_1+}^\alpha(L_p((T_1,T_2);R))
\end{eqnarray*}
if $\alpha p<1$. In addition, the mapping $r\mapsto D_{T_1+}^\alpha S(t-\cdot)G(u(\cdot))[r]$ is measurable on $[T_1,t]$ for $t\le T_2$ and satisfies the estimate
\begin{eqnarray*}
  ||D_{T_1+}^\alpha S(t-\cdot)G(u(\cdot))[r]||_{L_2(V)}\le c (1+\|u\|_{\beta,\sim})(r-T_1)^{-\alpha}\bigg(1+\frac{(r-T_1)^\beta}{(t-r)^\beta}\bigg).
\end{eqnarray*}
Similarly, assuming that $u\in C^{\beta}([T_1,T_2];V)$, then
\begin{eqnarray*}
  ||D_{T_1+}^\alpha S(t-\cdot)G(u(\cdot))[r]||_{L_2(V)}\le c (1+\|u\|_{\beta})(r-T_1)^{-\alpha}\bigg(1+(r-T_1)^\beta+\frac{(r-T_1)^\beta}{(t-r)^\beta}\bigg).
\end{eqnarray*}
In addition, $\zeta_{iT_2-}=(\omega_{T_2-}(t),e_i)\in I_{T_2-}^{1-\alpha}(L_{p^\prime}((T_1,T_2);R))$ for any $p^\prime>1$ and $\omega\in C^{\beta^\prime}([T_1,T_2]; V) $ for $\beta^\prime>1-\alpha$, and
$$|D_{T_2-}^{1-\alpha} \omega_{T_2-}[r]| \leq c |||\omega|||_{\beta^\prime,T_1,T_2} (T_2-r)^{\beta^\prime+\alpha-1}.$$
\end{lemma}

In the following result we remind the existence theorem and comment briefly why we have considered two different phase spaces.

\begin{theorem} (see \cite{ChGGSch12}) \label{t11}
If $u_0\in V$, then for every $T>0$ the equation (\ref{eq7}) has a unique solution $u$ in $C^{\beta,\sim}([0,T];V)$.\\
If $u_0\in V_\beta$, then for every $T>0$ the equation (\ref{eq7}) has a unique solution $u$ in $C^{\beta}([0,T];V)$.
\end{theorem}

Let us comment these results.
Existence of this kind of equations has been investigated in \cite{MasNua03} and \cite{GLS09} when considering as integrators of the stochastic integrals a fractional Brownian motion with Hurst parameter in $(1/2,1)$ and certain phase spaces not as natural as the space of H\"older continuous functions. However, in this article we present the existence theory in other function spaces, namely the space of H{\"o}lder continuous solutions for appropriate exponents. In fact, in \cite{ChGGSch12}, assuming that $1/2<\beta<\beta^\prime<H$ and $1-\beta^\prime<\alpha<\beta$, the existence of a unique pathwise mild solution has been obtained by applying the Banach fixed point theorem.

The main reason to consider the space $C^{\beta,\sim}([0,T];V)$ is that $t\mapsto S(t)u_0$ is not a $\beta$--H{\"o}lder-continuous function
but an element of that space. However, if $u_0\in V_\beta$ then $t\mapsto S(t)u_0$ is an element of $C^{\beta}([0,T];V)$.
Considering the equivalent norm from (\ref{eq38}) on  $C^{\beta,\sim}([0,T];V)$, we are able to adapt $\rho$ to the data of our problem (in particular to $T$) such that the right hand side of (\ref{eq7}) forms an operator which satisfies the conditions of the Banach fixed point theorem if $u_0\in V$. In particular, see  \cite{ChGGSch12}, we obtain estimates of the integral like
\begin{eqnarray}
\left\|\int_0^\cdot S(\cdot-r)G(u(r))d\omega\right\|_{\beta,\sim,0,T}&\le& c_T|||\omega|||_{\beta^\prime,0,T}(1+\|u\|_{\beta,\sim,0,T}), \label{eq36}\\
  \left\|\int_0^\cdot S(\cdot-r)G(u(r))d\omega\right\|_{\beta,0,T} &\le & c_T|||\omega|||_{\beta^\prime,0,T}(1+\|u\|_{\beta,0,T})\label{eq37}.
\end{eqnarray}
Assuming that $u_0\in V$ the corresponding solution $u$ of (\ref{eq7}) satisfies $u(T)\in V_\beta$ for every $T>0$, due to
\begin{eqnarray}\label{eq31}
|u(T)|_{V_\beta}\leq cT^{-\beta}|u_0|+c_T |||\omega|||_{\beta^\prime,0,T}(1+\|u\|_{\beta,\sim,0,T})<\infty.
\end{eqnarray}
The constant $c_T$ in the above formulas depends on the operators $S,\,G$ and is bounded when $T$ is bounded. $c_T$ can be chosen independently of $\omega$.\\

In addition to the previous estimates we also have the following one.
\begin{lemma}\label{l11}
For $1-\beta^\prime<\alpha<\beta$ and $u\in C^\beta([0,T];V)$ we have
\begin{eqnarray*}
  \left|\int_0^T S(T-r)G(u(r))d\omega\right|_{V_\beta} &\le & c_T|||\omega|||_{\beta^\prime,0,T}(1+\|u\|_{\beta,0,T}).
\end{eqnarray*}
\end{lemma}
\begin{proof} Take $\alpha^\prime>\alpha$ such that $\alpha^\prime+\beta<\alpha+\beta^\prime$, then the following inequalities hold
\begin{eqnarray*}
  &&\bigg|(-A)^\beta\int_0^T S(T-r)G(u(r))d\omega\bigg| \\
  &\le& \frac{1}{\Gamma(1-\alpha)}\int_0^T\bigg(\frac{|S(T-r)(-A)^\beta G(u(r))|}{r^\alpha} \\
   &+& \alpha \int_0^r\frac{|(S(T-r)-S(t-q))(-A)^\beta G(u(r))|}{(r-q)^{1+\alpha}}dq\\
   &+&\alpha \int_0^r\frac{|S(T-q)(-A)^\beta (G(u(r))-G(u(q)))|}{(r-q)^{1+\alpha}}dq\bigg)|||\omega|||_{\beta^\prime,0,T}(T-r)^{\alpha+\beta^\prime-1}dr \\
  &\le& c|||\omega|||_{\beta^\prime,0,T}(1+||u||_{\beta,0,T})\int_0^T\bigg(\frac{1}{r^\alpha(T-r)^\beta}\\
  &+&\int_0^r\frac{(r-q)^{\alpha^\prime}}{(T-r)^{\beta+\alpha^\prime}(r-q)^{1+\alpha}}dq+
  \int_0^r\frac{(r-q)^\beta}{(T-r)^\beta(r-q)^{1+\alpha}}dq\bigg)(T-r)^{\alpha+\beta^\prime-1}dr \\
  &\le & c_T |||\omega|||_{\beta^\prime,0,T}(1+||u||_{\beta,0,T})  \end{eqnarray*}
where the last inequality is true due to the choice made for the parameters. \quad \end{proof}

\subsection{Non--autonomous dynamical systems of (\ref{eq7})}\label{N-Ads}\label{s3}

According to Lemma \ref{l3} and Lemma \ref{l18}, for $t,\tau\in R^+$ yields
\begin{eqnarray*}
    \int_{\tau}^{t+\tau} S(t+\tau-r)G(u(r)) d\omega(r)=\int_{0}^{t} S(t-r)G(u(r+\tau)) d\theta_\tau \omega(r),
\end{eqnarray*}
where $S$ and $G$ satisfy the corresponding conditions of Lemma \ref{l18}. From this formula it follows easily that the solution of (\ref{eq6}) generates a non-autonomous dynamical system $\phi$ on $R^+$ with state space $V$. However, later on from $\phi$ we shall derive another non-autonomous dynamical system but with discrete time set $Z^+$.

Next we introduce some stopping times which are generated by elements of $\Omega$. These stopping times are needed to keep $|||\omega|||_{\beta^{\prime}}$ small, which is necessary to obtain appropriate a priori estimates for the solution. \\

In a first step, fix $\mu>0$ and define the stopping times as follows
\begin{eqnarray}
T(\omega)&=&\inf\{\tau>0:|||\omega|||_{\beta^\prime,0,\tau}+\mu\tau^{1-\beta^\prime}\ge\mu\},\nonumber \\[-1.5ex]
\label{eq29}\\[-1.5ex]
\hat T(\omega)&=&\sup\{\tau<0:|||\omega|||_{\beta^\prime,\tau,0}+\mu|\tau|^{1-\beta^\prime}\ge\mu\}.\nonumber
\end{eqnarray}

\begin{lemma}\label{l7}
For $\omega\in\Omega$ we have that $T(\omega),\,-\hat T(\omega)$ are in $(0,1]$. In addition, $T(\omega)=-\hat T(\theta_{T(\omega)}\omega)$ and $\hat T(\omega)=-T(\theta_{\hat T(\omega)}\omega)$.
\end{lemma}

\begin{proof}
It is easily seen that $|T(\omega)|,\,|\hat T(\omega)|\le 1$. Moreover, since $\omega$ has a finite $\beta^{\prime\prime}$--H{\"o}lder seminorm and $\beta^\prime<\beta^{\prime \prime}$ we have
  \[  \lim_{\tau\downarrow 0}|||\omega|||_{\beta^\prime,0,\tau}=0,\]
and, in addition,
\[     \lim_{\tau\to\infty}|||\omega|||_{\beta^\prime,0,\tau}+\mu \tau^{1-\beta^\prime}=\infty\]
where we suppose that this condition holds for any $\omega\in\Omega$.

Therefore, thanks to the intermediate value theorem, we only need to prove the continuity of the strictly increasing mapping $\tau \rightarrow |||\omega|||_{\beta^\prime,0,\tau}+\mu \tau^{1-\beta^\prime}$ to ensure that there is a time $\hat \tau_0$ such that $|||\omega|||_{\beta,0,\hat \tau_0}+\mu\hat \tau_0^{1-\beta^\prime}=\mu$, which means that we can replace the inequality of (\ref{eq29}) by an equality.

Fixed $\tau_0 >0$ and define $\omega^{\tau_0}$ given by
$$\omega^{\tau_0}(s)=\left\{
\begin{array}{lcl}
\omega(s)&:&{\rm for }\ s< \tau_0,\\
\omega(\tau_0)&:&{\rm for }\ s\geq \tau_0.
\end{array}\right.\\$$
Thus, for $\tau \geq \tau_0$,

\[ |||\omega|||_{\beta^\prime,0,\tau}-|||\omega|||_{\beta^\prime,0,\tau_0}
 =|||\omega|||_{\beta^\prime,0,\tau}-|||\omega^{\tau_0}|||_{\beta^\prime,0,\tau}\le |||\omega- \omega^{\tau_0}|||_{\beta^\prime,0,\tau}=|||\omega|||_{\beta^\prime,\tau_0,\tau},\]
and because
\[\lim_{\tau\downarrow \tau_0}|||\omega|||_{\beta^\prime,\tau_0,\tau}\leq     \lim_{\tau\downarrow \tau_0}|||\omega|||_{\beta^{\prime \prime},\tau_0,\tau} (\tau-\tau_0)^{\beta^{\prime \prime}-\beta^{\prime }} \leq     \lim_{\tau\downarrow \tau_0}|||\omega|||_{\beta^{\prime \prime},0,\tau} (\tau-\tau_0)^{\beta^{\prime \prime}-\beta^{\prime }} =0,\]
then the mentioned continuity property is true. Note that we could obtain a similar conclusion when taking $\tau\uparrow \tau_0$. Moreover, because $T\mapsto |||\omega|||_{\beta^\prime,0,T}+T^{1-\beta^\prime}$ is {\em strictly} increasing  and $T\mapsto |||\omega|||_{\beta^\prime,T,0}+|T|^{1-\beta^\prime}$ is {\em strictly} decreasing, it is easily seen that

\begin{eqnarray*}
   \mu= |||\omega|||_{\beta^\prime,0,T(\omega)}+\mu T(\omega)^{1-\beta^\prime}=|||\theta_{T(\omega)}\omega|||_{\beta^\prime,\hat T(\theta_{T(\omega)}\omega),0}+\mu (-\hat T(\theta_{T(\omega)}\omega))^{1-\beta^\prime}
\end{eqnarray*}
which is only possible if $T(\omega)=-\hat T(\theta_{T(\omega)}\omega)$.\qquad\end{proof}

From $T(\omega),\,\hat T(\omega)$ we derive a sequence of  stopping times. For $\omega\in \Omega$ we define
\begin{eqnarray}\label{stop}
    T_i(\omega)=\left\{\begin{array}{lcr}0&:&i=0,\\T_{i-1}(\omega)+T(\theta_{T_{i-1}(\omega)}\omega)&:&i\in N,\\
    T_{i+1}(\omega)+\hat T(\theta_{T_{i+1}(\omega)}\omega)&:&i\in -N.\end{array}\right.
\end{eqnarray}
Then $(T_i(\omega))_{i\in Z}$ satisfies the cocycle property: for $i,\,j\in Z$ we have
\[
    T_0(\omega)=0,\qquad T_i(\omega)+T_j(\theta_{T_i(\omega)}\omega)=T_{i+j}(\omega).
\]
Note that with the previous notation we are identifying $T(\omega)$ with $T_1(\omega)$ and $\hat T(\omega)=T_{-1}(\omega)$.

Thanks to the definition, the stopping times enjoy the following order property:

\begin{lemma}\label{l12}
Let $t_1\le t_2$. Then
$$
t_1+ \hat T(\theta_{t_1}\omega)\le t_2+ \hat T(\theta_{t_2}\omega).
$$
\end{lemma}
\begin{proof}
Suppose that the contrary inequality holds, that means that $\hat T(\theta_{t_1}\omega)>\hat T(\theta_{t_2}\omega)$. Therefore, $(-\hat T(\theta_{t_1}\omega))^{1-\beta^\prime}<(-\hat T(\theta_{t_2}\omega))^{1-\beta^\prime}$ and moreover
\begin{eqnarray*}
    \mu&=&|||\theta_{t_2}\omega|||_{\beta^\prime,\hat T(\theta_{t_2}\omega),0}+\mu(-\hat T(\theta_{t_2}\omega))^{1-\beta^\prime}\\
   & =&\sup_{\hat T(\theta_{t_2}\omega)+t_2\le s < t\le t_2}\frac{|\omega(t)-\omega(s)|}{|t-s|^{\beta^\prime}}+\mu(-\hat T(\theta_{t_2}\omega))^{1-\beta^\prime}\\
    &>&\sup_{\hat T(\theta_{t_1}\omega)+t_1\le s < t\le t_1}\frac{|\omega(t)-\omega(s)|}{|t-s|^{\beta^\prime}}+\mu(-\hat T(\theta_{t_1}\omega))^{1-\beta^\prime}\\
    &=&|||\theta_{t_1}\omega|||_{\beta^\prime,\hat T(\theta_{t_1}\omega),0}+\mu(-\hat T(\theta_{t_1}\omega))^{1-\beta^\prime}=\mu.
\end{eqnarray*}
However, this chain of inequalities causes a contradiction.\qquad\end{proof}

\vskip.5cm

If $t_2+\hat T(\theta_{t_2}\omega) \leq t_1 \leq t_2$, then iterating the formula in the previous lemma we obtain
\begin{eqnarray}
\cdots&\le& t_1+ \hat T(\theta_{t_1}\omega)+ \hat
T(\theta_{t_1+ \hat T(\theta_{t_1}\omega)}\omega) \nonumber \\
&\le& t_2+ \hat T(\theta_{t_2}\omega)
    + \hat T(\theta_{t_2+ \hat
    T(\theta_{t_2}\omega)}\omega) \label{eq33} \\
    &   \le & t_1+ \hat T(\theta_{t_1}\omega)\le t_2+ \hat
T(\theta_{t_2}\omega)\le
 t_1\le t_2. \nonumber
\end{eqnarray}

\subsection{Global attractors for the non-autonomous dynamical systems associated to (\ref{eq7})}

As an preparation of the next key result, Lemma \ref{l9-u} below, we formulate the following Gronwall-like lemma.
\begin{lemma}\label{l10} Let  $\lambda,\,v_0,\,k_0,\,k_1<1,\,k_{2}$ positive numbers and let $(t_i)_{i\in Z^+}$ be a sequence of positive numbers, with $t_0=0$, such that
\begin{eqnarray}\label{ti}
t_{i-1}-t_{i-2}\leq -\frac{2}{\lambda}\log k_1,\quad {\rm for }\ i\geq 2.
\end{eqnarray}
Suppose that  for a sequence of positive numbers $(U_i)_{i\in N}$ the following inequalities hold true:
\begin{eqnarray}
    U_{i}&\le& k_0 v_0e^{-\lambda t_{i-1}}+\sum_{m=1}^{i-1}k_1U_me^{-\lambda(t_{i-1}-t_m)}\nonumber \\[-1.5ex]
\label{eq17}\\[-1.5ex]
    &+&\sum_{m=1}^{i-1} e^{-\lambda(t_{i-1}-t_m)} k_2 +k_2\footnotemark,\quad i=1,2,3,\cdots.\nonumber
\end{eqnarray}
\footnotetext{The sum $\sum _{m=1}^{i-1}$ is assumed to be zero for $i=1$.}
Then we have
\begin{eqnarray}
U_i &\le& (k_0v_0+k_2)(1+k_1)^{i-1}e^{-\frac{\lambda}{2} t_{i-1}} \nonumber \\[-1.5ex]
\label{eq18}\\[-1.5ex]
&+&\sum_{m=1}^{i-1} 2k_2(1+k_1)^{i-1-m}e^{-\frac{\lambda}{2}(t_{i-1}-t_{m})},\quad {\rm for }\ i=1,2,3,\cdots.\nonumber
\end{eqnarray}
\end{lemma}

{\em Proof.} First of all, note that the inequality
\begin{eqnarray}\label{exp}
k_1+e^{-x}\le (1+k_1)e^{-\frac{x}{2}}
\end{eqnarray}
holds true for $x\in [0,-2\log k_1]$.

Denote the right hand side of (\ref{eq17}) by $Z_i$ and the right hand side of (\ref{eq18}) by $S_i$. We want to prove that $U_{i}\le S_{i}$ for all $i\ge 1$, for which it is enough to prove by induction that $Z_{i}\le S_{i}$.

For $i=1$, from (\ref{eq17}) we have $U_1\leq Z_1=k_0v_0+k_2$.  From (\ref{eq18}) we have $S_1=k_0v_0+k_2$, so $Z_1=S_1$.

For $i=2$, from (\ref{eq17}) we have $Z_2=k_0v_0e^{-\lambda t_1}+k_1U_1+2k_2$. From (\ref{eq18}) we have $S_2=(k_0v_0+k_2)(1+k_1) e^{-\frac{\lambda}{2} t_{1}}+2k_2$. It suffices to take into account that from (\ref{exp}) we have
\[k_1+e^{-\lambda t_1}\le (1+k_1)e^{-\frac{\lambda}{2} t_1},\]
since by (\ref{ti}) in particular $\lambda t_{1}\leq -2\log k_1$, Therefore,
\begin{eqnarray*}Z_2&=&k_0v_0 e^{-\lambda t_1} +k_1U_1+2k_2 \leq (k_0v_0+k_2)e^{-\lambda t_1}+k_1(k_0v_0+k_2) +2k_2 \\
&\leq& (k_0v_0+k_2)(1+k_1)e^{-\frac{\lambda }{2}t_1}+2k_2=S_2.
\end{eqnarray*}
For $i\geq 3$ we apply induction. Firstly note that
\[e^{-\lambda(t_{i-1}-t_{i-2})}Z_{i-1}=Z_{i}-k_1U_{i-1}-2k_2+k_2 e^{-\lambda(t_{i-1}-t_{i-2})}\]
and thereby, assuming that $Z_{i-1}\le S_{i-1}$, thanks to (\ref{ti}) and (\ref{exp}),
\begin{eqnarray*}
  U_{i}&\le&  Z_{i} \le e^{-\lambda(t_{i-1}-t_{i-2})}Z_{i-1}+k_1U_{i-1}+2k_2\le (e^{-\lambda(t_{i-1}-t_{i-2})}+k_1)Z_{i-1}+2k_2\\
  &\le &(e^{-\lambda(t_{i-1}-t_{i-2})}+k_1)S_{i-1}+2k_2\leq (1+k_1)e^{-\frac{\lambda}{2}(t_{i-1}-t_{i-2})}S_{i-1}+2k_2=S_{i}. \quad \endproof
\end{eqnarray*}

Next we obtain an appropriate a priori estimate for the solution of our equation when assuming that the initial condition is regular, namely $u_0\in V_\beta$. Later on this a priori estimate will be the key to obtain an absorbing set, the main ingredient to ensure the existence of a pullback attractor for the system.

In what follows, $c$ shall denote a positive constant which value is unimportant and can of course change from line to line, and may depend on $S$ and $G$ but not on $\omega$.
\begin{lemma}\label{l9-u}
For $\omega\in\Omega$ let $u$ be a solution of (\ref{eq6}) where $u_0\in V_\beta$ and let $(T_i(\theta_{T_j(\omega)}\omega))_{i\in Z}$ be the sequence of stopping times defined at the beginning of this section. Then we have that
\begin{eqnarray*}
    \|u\|_{\beta,T_{i-1}(\theta_{T_j}\omega),T_{i}(\theta_{T_j}\omega)}&\le& c e^{-\lambda_1 T_{i-1}(\theta_{T_j}\omega)}|u_0|_{V_\beta}+c\mu\|u\|_{\beta,T_{i-1}(\theta_{T_j}\omega),T_{i}(\theta_{T_j}\omega)}\\
    &+&\sum_{m=1}^{i-1}c\mu e^{-\lambda_1(T_{i-1}(\theta_{T_j}\omega)-T_m(\theta_{T_j}\omega))}||u||_{\beta,T_{m-1}(\theta_{T_j}\omega),T_{m}(\theta_{T_j}\omega)}\\
    &+&\sum_{m=1}^{i-1}e^{-\lambda_1(T_{i-1}(\theta_{T_j}\omega)-T_m(\theta_{T_j}\omega))}c\mu+ c\mu
\end{eqnarray*}
where $T_j=T_j(\omega)$, and $\lambda_1$ is the smallest eigenvalue of $-A$.\\
\end{lemma}

The proof of this lemma can be found in the Appendix Section.\\

The cocycle property of $(T_i(\omega))_{i\in Z}$ allows us to introduce a discrete
non-autonomous dynamical system $\Phi$ on $V$ with time set $T^+=Z^+$ for {\em every} $\omega\in\Omega$. In particular, we consider the new shift given by
$$\tilde \theta:  Z \times Z \rightarrow  Z$$
defined by $\tilde \theta_i j=i+j$, for $i,j\in  Z$, where the set $\Omega$ in the general definition of a flow is identified here with $Z$. Then we define
$$\Phi:  Z^+\times Z \times \Omega \times V \rightarrow V$$
as
\begin{eqnarray*}
\Phi(i,j,\omega,u_0)&=&S(T_i(\theta_{T_j(\omega)}\omega))u_0+\int_0^{T_i(\theta_{T_j(\omega)}\omega)} S(T_i(\theta_{T_j(\omega)}\omega)-r)G(u(r)) d\theta_{T_j(\omega)}\omega\\
&=&\phi(T_i(\theta_{T_j(\omega)}\omega),\theta_{T_j(\omega)}\omega,u_0).
\end{eqnarray*}
Note that $\Phi(i,j,\omega,u_0)$ is given by the solution $u$ of (\ref{eq7}) for the noise path $\theta_{T_j(\omega)}\omega$ at time $T_i(\theta_{T_j(\omega)}\omega)$. We would like to emphasize that, in the definition of $\Phi$, $\omega$ acts as a parameter.\\

We now specify the constant $\mu$. Let $c>0$ be the constant from Lemma \ref{l9-u}.  We choose $\mu$ sufficiently small such that for $k_1(\mu)=c\mu/(1-c\mu)<1$ the following inequality holds:
\begin{eqnarray}\label{eq34}
1<-\frac{2}{\lambda_1}\log k_1(\mu).
\end{eqnarray}
Let us also define
\[\lambda=\lambda_1,\,k_0=\frac{c}{1-c\mu},\,k_2=k_1=k_1(\mu),\,t_i=T_i(\theta_{T_j(\omega)}\omega).\]

As we will show in Corollary \ref{l9} below, the choice done in (\ref{eq34}) ensures the condition (\ref{ti}), and with it we will prove the existence of an absorbing ball for $\Phi$ (see Lemma \ref{l13}).

\begin{corollary}\label{l9}
Let $u_0\in V_\beta$ and suppose that $\mu$ is chosen such that (\ref{eq34}) is satisfied. Then the following inequality holds true\footnote{Again, the sum $\sum _{m=1}^{i-1}$ is assumed to be zero for $i=1$.}
\begin{eqnarray}
    |\Phi(i,j,\omega,u_0)|&\le &(1+k_1)^{i-1}e^{-\frac{\lambda_1}{2} T_{i-1}(\theta_{T_j(\omega)}\omega)}(k_0|u_0|_{V_\beta}+k_2) \nonumber \\[-1.5ex]
    \label{eq22}\\[-1.5ex]
   & +&\sum_{m=1}^{i-1} 2k_2(1+k_1)^{i-1-m} e^{-\frac{\lambda_1}{2}(T_{i-1}(\theta_{T_j(\omega)}\omega)-T_{m}(\theta_{T_j(\omega)}\omega))}.\nonumber
   \end{eqnarray}
\end{corollary}

{\em Proof.} We only need to take into account that, by definition,
\[ |\Phi(i,j,\omega,u_0)|=|u(T_i(\theta_{T_j(\omega)}\omega)|\le  \|u\|_{\beta,T_{i-1}(\theta_{T_j(\omega)}\omega),T_{i}(\theta_{T_j(\omega)}\omega)}\]
and therefore, thanks to Lemma \ref{l9-u}, we obtain the desired estimate due to the above choice of $k_0,\, k_1,\,k_2,\,T_i$ if in addition we take $v_0=|u_0|_{V_\beta}$.
We would like to stress that we can apply the Gronwall-like Lemma \ref{l10} in this situation since the stopping times satisfy the condition (\ref{ti}). Actually, the cocycle property and (\ref{eq34}) imply
$$
 t_{i-1}-t_{i-2}=T_{i-1}(\theta_{T_j(\omega)}\omega)-T_{i-2}(\theta_{T_j(\omega)}\omega)=T_{1}(\theta_{T_{i-2+j}(\omega)}\omega)\le 1<-\frac{2}{\lambda_1}\log k_1(\mu). \eqno\endproof $$

Now we formulate a smallness condition for all $\omega\in\Omega$.
We  assume that the stopping times satisfy
\begin{eqnarray}\label{eq21}
   1>  \liminf_{i\to-\infty}\frac{|T_i(\omega)|}{|i|} & \geq d \geq \frac{2 (\log(1+k_1(\mu))+\nu)}{\lambda_1}
\end{eqnarray}
where $\nu\in [0,\frac{ d\lambda_1}{2})$ is the parameter describing backward $\nu$-exponentially growing sets $\dD_{Z,V}^\nu$. 
In addition assume that
\begin{eqnarray}\label{eq39}
\nu+d>1.
\end{eqnarray}

We also assume that the sequence $(|T(\theta_{T_i(\omega)}\omega)|^{-\beta})_{i\in Z}$ is subexponential growing for $\omega \in\Omega$:
\begin{eqnarray}\label{eq39b}
\lim_{i\to-\infty}\frac{\log^+(|T(\theta_{T_i(\omega)}\omega)|^{-\beta})}{|i|}=0.
\end{eqnarray}
Later on, in Section 4, we will give an example of a set $\Omega$ of $\omega$ fulfilling conditions (\ref{eq21}), (\ref{eq39}) and (\ref{eq39b}) being $(\theta_t)_{t\in R}$--invariant.\\

We now consider the discrete non--autonomous dynamical system $\Phi$ with set of non-autonomous perturbations $Z$ and  shifts $\tilde \theta_i j=i+j$. Recall that the set system $\dD_{Z,V}^\nu$ is given by the family of sets $(D(i))_{i\in Z}$ such that $D(i)\subset V$ is included
in a ball with center zero and radius $r(i)$  where
\begin{eqnarray*}
    \limsup_{i\to-\infty} \frac{\log^+r(i)}{|i|}<\nu.
\end{eqnarray*}

Our next aim is to prove that the discrete non-autonomous dynamical system $\Phi$ has an absorbing set  consisting in a ball $B$ contained in $\mathcal D_{Z,V}^\nu$, which means that in particular $B$ is in $V$.\\

\begin{lemma}\label{l13}
Suppose that (\ref{eq34}), (\ref{eq21}), (\ref{eq39}) and (\ref{eq39b}) hold.
Then $\Phi$ has a $\dD_{Z,V}^\nu$-absorbing set $B(\omega)=(B(i,\omega))_{i\in Z}$, where $B(i,\omega)$ is
given by a ball in $V$ with center 0 and radius
\begin{eqnarray}\label{eq27}
    R(i,\omega)=2\sum_{m=-\infty}^0 2k_2(1+k_1)^{-m} e^{\frac{\lambda_1}{2}T_{m}
    (\theta_{T_{i-1}(\omega)}\omega)}.
\end{eqnarray}
\end{lemma}

\begin{proof} We want to get a $\dD_{Z,V}^\nu$-absorbing set $B$ for $\Phi$. In Corollary \ref{l9} we have obtained estimates for $\Phi$ when the initial condition $u_0\in V_\beta$, so in this proof in a first step we have to ensure that picking any $D\in \mathcal D_{Z,V}^\nu$ we can build an appropriate set $F$ depending on $D$ such that $F\in \mathcal D_{Z,{V_\beta}}^\nu$. This property will be the key to later on proving that $B$ is an absorbing set for $\Phi$.

Define for  $D\in\dD_{Z,V}^\nu$ the set
$$F(j,\omega):=\overline{\Phi(1,j-1,\omega,D(j-1))}.$$
Such an $F$ is a backward $\nu$--exponentially growing set in $V_\beta$, since, if $v_0\in D(j-1)$ we know that
\begin{eqnarray}
\qquad \Phi(1,j-1,\omega,v_0)&=&u(T_1(\theta_{T_{j-1}(\omega)}\omega))=S(T_1(\theta_{T_{j-1}(\omega)}\omega))v_0\nonumber \\[-1.5ex]
\label{eqv}\\[-1.5ex]
&+&\int_0^{T_1(\theta_{T_{j-1}(\omega)}\omega)} S(T_1(\theta_{T_{j-1}(\omega)}\omega)-r) G(u(r))d\theta_{T_{j-1}(\omega)}\omega.\nonumber
\end{eqnarray}
From (\ref{eq31}), 
\begin{eqnarray}
&& |\Phi(1,j-1,\omega,v_0)|_{V_\beta} \leq \frac{c}{|T_1(\theta_{T_{j-1}(\omega)} \omega)|^\beta} |v_0|\nonumber \\[-1.5ex]
\label{eq32}\\[-1.5ex]
&& \qquad +c|||\omega|||_{\beta^\prime,0,T_1(\theta_{T_{j-1}(\omega)} \omega)} (1+\|u\|_{\beta,\sim, 0, T_1(\theta_{T_{j-1}(\omega)} \omega)}).\nonumber
\end{eqnarray}
The last term on the previous expression can be estimated by the technique of Lemma \ref{l9-u} for $i=1$ but using the $\|\cdot\|_{\beta,\sim, 0, T_1(\theta_{T_{j-1}(\omega)} \omega)}$--norm together with (\ref{eq36}), getting that
\[
(1-c\mu)\|u\|_{\beta,\sim, 0, T_1(\theta_{T_{j-1}(\omega)} \omega)} \leq c |v_0|+c\mu
\]
and therefore,
\begin{eqnarray*}
\sup_{v_0\in D(j-1,\omega)}&&|\Phi(1,j-1,\omega,v_0)|_{V_\beta}\\
&&\leq \bigg(\frac{c}{|T_1(\theta_{T_{j-1}(\omega)} \omega)|^\beta}+\frac{c\mu}{1- c\mu}\bigg)\sup_{v_0\in D(j-1,\omega)} |v_0|
+ c\mu \bigg(1+ \frac{ c\mu}{1- c\mu}  \bigg).
\end{eqnarray*}

The first term on the right hand side is backward $\nu$--exponentially growing, which follows from the assumption that $(|T(\theta_{T_i(\omega)}\omega)|^{-\beta})_{i\in Z}$ is subexponentially growing.
Indeed it is a product of two terms where one factor satisfies (\ref{eq11}) while the other is $\nu$--exponentially growing.

Therefore, for any $D\in\dD_{Z,V}^\nu$ the set $F(j,\omega)\in \dD_{Z,V_\beta}^\nu$ and then there is a sequence $(B_{V_\beta}(0,\rho(i,\omega)))_{i\in Z}$ backward $\nu$--exponentially growing such that these balls in $V_\beta$ with center zero and radius $\rho(i,\omega)$ contain the sets $F(i,\omega)$. Moreover, from Corollary \ref{l9} we immediately have that
\begin{eqnarray*}
    \sup_{u_0\in F(j-i,\omega)}|\Phi(i,j-i,\omega,u_0)|&\le &(1+k_1)^{i-1}e^{-\frac{\lambda_1}{2} T_{i-1}(\theta_{T_{j-i}(\omega)}\omega)}(k_0\rho(j-i,\omega)+k_2)\\
   & +&\sum_{m=1}^{i-1} 2k_2(1+k_1)^{i-1-m} e^{-\frac{\lambda_1}{2}(T_{i-1}(\theta_{T_{j-i}(\omega)}\omega)-T_{m}(\theta_{T_{j-i}(\omega)}\omega))}.
\end{eqnarray*}
Taking into account that the stopping times satisfy the cocycle property, it holds
\[
    T_{i-1}(\theta_{T_{j-i}(\omega)}\omega)-T_{m}(\theta_{T_{j-i}(\omega)}\omega)
    =T_{i-1-m}(\theta_{T_{j+m-i}(\omega)}\omega)=-T_{-i+m+1}(\theta_{T_{j-1}(\omega)}\omega)\]
and thus
\begin{eqnarray}
    |\Phi(i,j-i,\omega,u_0)|&\le &(1+k_1)^{i-1}e^{\frac{\lambda_1}{2} T_{-i+1}(\theta_{T_{j-1}(\omega)}\omega)}(k_0\rho(j-i,\omega)+k_2)\nonumber \\[-0.5ex]
    &+&\sum_{m=1}^{i-1} 2k_2(1+k_1)^{i-1-m} e^{\frac{\lambda_1}{2}T_{-i+m+1}(\theta_{T_{j-1}(\omega)}\omega)} \nonumber \\[-2.5ex]
    \label{eqphi}\\[-1.5ex]
    &=&(1+k_1)^{i-1}e^{\frac{\lambda_1}{2} T_{-i+1}(\theta_{T_{j-1}(\omega)}\omega)}(k_0\rho(j-i,\omega)+k_2)\nonumber \\[-0.5ex]
    &+&\sum_{m=2-i}^{0} 2k_2(1+k_1)^{-m} e^{\frac{\lambda_1}{2}T_{m}(\theta_{T_{j-1}(\omega)}\omega)}\nonumber.
\end{eqnarray}
Note that for every $\eps>0$ there is an $m_\epsilon>0$ such that for $m<0,\,|m|>m_\epsilon$, then $|T_{m}(\theta_{T_{j-1}(\omega)}\omega)| > (d-\epsilon)|m|$, or equivalently, $\frac{\lambda_1}{2} T_{m}(\theta_{T_{j-1}(\omega)}\omega) < m \frac{\lambda_1}{2} (d- \epsilon)$, which is a consequence of the first inequality in (\ref{eq21}). Therefore,
\begin{eqnarray*}
(1+k_1)^{-m} e^{\frac{\lambda_1}{2}T_{m}(\theta_{T_{j-1}(\omega)}\omega)}
&=&e^{-m\log(1+k_1)+\frac{\lambda_1}{2} T_{m}(\theta_{T_{j-1}(\omega)}\omega)} \\
&\leq & e^{m(-\log(1+k_1)+\frac{\lambda_1}{2} (d-\epsilon))}\le e^{\nu m}
\end{eqnarray*}
where the last estimate is true since $\frac{\lambda _1}{2} d -\log (1+k_1)-\nu- \frac{\lambda _1}{2}  \epsilon>0$ for small $\epsilon>0$, which follows from the second inequality of (\ref{eq21}). Hence, the sum on the right hand side of (\ref{eqphi}) converges to $\frac{R(j)}{2}$ when $i\to\infty$.

On the other hand, the first term on (\ref{eqphi}) converges to zero for $i\to\infty$ due to (\ref{eq21}), since
\begin{eqnarray*}
   & & (1+k_1)^{i-1}e^{\frac{\lambda_1}{2} T_{-i+1}(\theta_{T_{j-1}(\omega)}\omega)}\sup_{u_0\in F(-i+j,\omega)}(k_0|u_0|_{V_\beta}+k_2)\\
    &
  \le&(1+k_1)^{i-1}e^{\frac{\lambda_1}{2} T_{-i+1}(\theta_{T_{j-1}(\omega)}\omega)}(k_0 \rho(j-i,\omega)+k_2).
\end{eqnarray*}
To obtain the absorbing property it remains to mention that, thanks to the cocycle property for $\Phi$, for sufficient large  $i\in Z$,
\begin{eqnarray*}
& & |\Phi(i+1,j-i-1,\omega,v_0)|=|\Phi(i,j-i,\omega,\Phi(1,j-i-1,\omega,v_0))|
\\&\leq& (1+k_1)^{i}e^{\frac{\lambda_1}{2} T_{i}(\theta_{T_{j-i-1}(\omega)}\omega)} (k_0 |\Phi(1,j-i-1,\omega,v_0)|_{V_\beta} +k_2) +R(j)/2,
\end{eqnarray*}
hence
\[
    \sup_{v_0\in D(j-i-1)}|\Phi(i+1,j-i-1,\omega,v_0)|\le \sup_{u_0\in F(j-i,\omega)}|\Phi(i,j-i,\omega,u_0)|< R(j)
\]
for sufficient large  $i\in Z$.\qquad 
\end{proof}\\

\begin{lemma}\label{l13-bis}
Under the conditions of Lemma \ref{l13}, the absorbing set $B(\omega)$ given by (\ref{eq27}) is contained in $\dD_{Z,V}^\nu$.
\end{lemma}

\begin{proof} We show that $(R(i))_{i\in N}$ is $\nu$--exponentially growing,  for which we will use the first inequality in (\ref{eq21}). Since
\[
    T_m(\theta_{T_{i-1}(\omega)}\omega)=T_{m+i-1}(\omega)-T_{i-1}(\omega)=T_{m+i-1}(\omega)+T_{-i+1}(\theta_{T_{i-1}(\omega)}\omega),
\]
we obtain that
\begin{eqnarray*}
   R(i)&=& 2\sum_{m=-\infty}^0 2k_2(1+k_1)^{-m} e^{\frac{\lambda_1}{2}T_{m}
    (\theta_{T_{i-1}(\omega)}\omega)}\\
    &=&2\sum_{m=-\infty}^0 2k_2(1+k_1)^{-m}e^{\frac{\lambda_1}{2} (T_{m+i-1}(\omega)+T_{-i+1}(\theta_{T_{i-1}(\omega)}\omega))}.
\end{eqnarray*}
Furthermore, due to (\ref{eq21}), for any sufficiently small $\delta$ there exists $ i_0(\delta)\in Z^-$ such that for $i\le i_0(\delta)$ we have
$$
    T_i(\omega)\le d i+\frac{\delta}{2}|i|,
$$
which together with $T_{-i+1}(\theta_{T_{i-1}(\omega)}\omega)\leq -i+1$ implies that
\begin{eqnarray*}
   R(i)&\le&2\sum_{m=-\infty}^0 2k_2(1+k_1)^{-m}e^{\frac{\lambda_1}{2} ((T_{m+i-1}(\omega)-d(i+m-1))+d(i+m-1)-(i-1))}\\
     & \le &2 \sum_{m=-\infty}^0  2k_2 (1+k_1)^{-m}e^{\frac{\lambda_1}{2} (\delta|i+m-1|+d(i+m-1)-(i-1))}\\
    & \le & 2 e^{\frac{\lambda_1}{2} \delta|i-1|}e^{(1-d)|i-1|}\sum_{m=-\infty}^0 2k_2(1+k_1)^{-m}e^{\frac{\lambda_1}{2} (\delta|m|+dm)}.
\end{eqnarray*}
For sufficiently small $\delta$ the sum on the right hand side is finite. Moreover, for $\nu> (1-d)$, see (\ref{eq39}), and $\delta>0$ small
\begin{eqnarray*}
    \lim_{i\to-\infty}e^{\nu i}R(i) &\leq & 2 e^{\frac{\lambda_1}{2} \delta}  \lim_{i\to-\infty}e^{\frac{\lambda_1}{2} \delta|i|+\nu i}e^{(1-d)|i-1|}\sum_{m=-\infty}^0 2k_2 (1+k_1)^{-m}e^{\frac{\lambda_1}{2} (\delta|m|+dm)} =0
\end{eqnarray*}
such that $B\subset \mathcal D_{Z,V}^\nu$.
\qquad 
\end{proof}
\\


The conditions of the last two lemmata are always grantable if the non--autonomous perturbation is small in the following sense: $\mu$ could be chosen small. Then, if $\nu>0$ is also small, $d$ should satisfy essentially the next three conditions:
\[
d<1,\qquad d+\nu>1,\qquad d>2\nu/\lambda_1.
\]
And we can find always an appropriate $d$ solving  these three inequalities. Note that the worst case happens when the first eigenvalue $\lambda_1$ is small, which forces $\nu$ to be chosen small enough and $d$ to be close to one. But $d$ close to one means that the stopping times $T_i$ are close to one in {\it the average}, in the sense of the first formula in (\ref{eq21}), or in other words, that the contribution of $|||\omega|||_{\beta^\prime,0,\tau}$ for the construction of the stopping time is small in the average, see (\ref{eq29}).\\

\begin{theorem}\label{t4}
Consider $\omega\in\Omega$ such that the assumptions of Lemma \ref{l13} hold.
Then the discrete non-autonomous dynamical system $\Phi(\cdot,\omega)$ has a pullback attractor $\{\mathcal A(i,\omega)\}_{i\in
 Z}$ with respect to the system of $\nu$--exponentially growing sets $\dD_{Z,V}^\nu$.
\end{theorem}

\begin{proof}
Note that $\Phi(i,j,\omega,\cdot)$ is continuous on $V$. Let  $B(\omega)\subset \dD_{Z,V}^\nu$ be the absorbing set from Lemma \ref{l13}. For any $j\in Z$, let $\tilde T(B,j,\omega)$ be the absorption time of $B(j,\omega)\in \dD_{Z,V}^\nu$ by itself. Set $T^*=T^*(j)=\tilde T(B,j,\omega)$. Notice that $\Phi(T^*,-T^*-1+j,\omega, B(-T^*-1+j,\omega))$ is absorbing, and set
\begin{eqnarray*}
   C(j,\omega):= \overline{\Phi(1,-1+j,\omega,\Phi(T^*,-T^*-1+j,\omega, B(-T^*-1+j,\omega)))}\subset B(j,\omega).
\end{eqnarray*}
Hence, the definition of $T^*$ ensures that $C(\omega)=(C(j,\omega))_{j\in Z}$ is $\dD_{Z,V}^\nu$--pullback absorbing and contained in $\dD_{Z,V}^\nu$. In addition, by (\ref{eq31}) and the compact embedding of $V_\beta\subset V$ the sets $C(j,\omega)$ are also compact. Now we can apply Theorem \ref{t3} giving the existence of a pullback attractor  $\{\mathcal A(i,\omega)\}_{i\in  Z}$ for $\Phi$. \qquad \end{proof}\\

The conclusion of the last theorem will be used to study $\phi$ as a non--autonomous dynamical system. In particular we show that this dynamical system has a pullback--attractor.

\subsection{Attracting sets for the non-autonomous dynamical systems associated to (\ref{eq7})}

We now study the non-autonomous dynamical system $\phi$ given by (\ref{eq7}). For this purpose it would be enough to consider this mapping $\phi$ along one orbit $\bigcup_{t\in R}\{\theta_{t}\omega\}$ for a fixed arbitrary $\omega\in\Omega$. However we are going to consider measurable mappings $\phi$ on the entire set $\Omega$.
In this sense the following definition is given.

Let us now define the family of closed non-empty sets $(D(\omega))_{\omega \in \Omega}\subset \dD_{R,V}^0$ with $D(\omega)=B_V(0,r(\omega))$ such that
\begin{eqnarray*}
    \limsup_{R^\ni t\to-\infty}\frac{\log^+ r(\theta_t\omega)}{|t|}=0.
\end{eqnarray*}
Define for such a $D$ the sets $(G(i,\omega))_{i\in Z}$ given by
\begin{equation}\label{eq40b}
G(i,\omega)=\overline{\bigcup_{\tau\in [T_{i-1}(\omega),T_i(\omega)]}D(\theta_\tau \omega)},\quad i\in Z,
\end{equation}
which are elements of $\dD_{Z,V}^0$. In the contrary case, we would find an $\omega$, a subsequence $(i^\prime(i,\omega))_{i\in Z^-}$, $i^\prime(i,\omega)\in Z^-$, $t_{i^\prime}\in
[T_{i^\prime-1}(\omega),T_{i^\prime}(\omega)]$ and $u_{i^\prime}\in D(\theta_{t_{i^\prime} }\omega)$ such that
\[
0<\limsup_{i^\prime\to-\infty}\frac{\log^+|u_{i^\prime}|}{|i^\prime|}.
\]
But we can estimate this expression by
\begin{equation}\label{eq41}
\limsup_{i^\prime\to-\infty}\frac{\log^+|u_{i^\prime}|}{|T_{i^\prime}(\omega)|}\limsup_{i^\prime\to-\infty}\frac{|T_{i^\prime}(\omega)|}{|i^\prime|}\leq \limsup_{i^\prime\to-\infty}\frac{\log^+|u_{i^\prime}|}{|t_{i^\prime}|}\limsup_{i^\prime\to-\infty}\frac{|T_{i^\prime}(\omega)|}{|i^\prime|}=0,
\end{equation}
due to the fact that $u_{i^\prime}\in D(\theta_{t_{i^\prime} }\omega)$. Let us emphasize that the last factor in the last term is estimated by one by definition of the stopping times.\\

We then can conclude:
\begin{lemma}\label{l15}
For $D\in\dD_{R,V}^0$ we have that
\begin{eqnarray}
    E_D(i,\omega)&=&\overline{\bigcup_{-\tau\in [\hat T(\theta_{T_i(\omega)}\omega),0]}\bigcup_{u_0\in D(\theta_{-\tau}\theta_{T_i(\omega)}\omega)}
    \{\phi(\tau,\theta_{-\tau}\theta_{T_i(\omega)}\omega,u_0)\}}
\end{eqnarray}
defines a set in $\dD_{Z,V}^0$.
In addition,
\[
H_D(i,\omega)=\overline{\bigcup_{\tau\in [0,T(\theta_{T_i(\omega)}\omega)]}\bigcup_{u_0\in D(\theta_{T_i(\omega)}\omega)}
    \{\phi(\tau,\theta_{T_i(\omega)}\omega,u_0)\}}
\]
is in $\dD_{Z,V}^\nu$ provided $D\in\dD_{Z,V}^\nu$.
\end{lemma}
\begin{proof} First of all, similar to the proof of Lemma \ref{l13} we obtain that
\[
(1-c\mu)\|u\|_{\beta,\sim, 0, T(\theta_{T_{j-1}(\omega)} \omega)}\le c( |u_0|+\mu ).
\]
On the other hand, for $-s\in [\hat T(\theta_{T_i(\omega)}\omega),0]$, due to  (\ref{eq36}) we have
\begin{eqnarray*}
    |\phi(s,\theta_{-s}\theta_{T_{i}(\omega)}\omega,u_0)|&\le &c |u_0|+\bigg|\int_0^s S(s-r)G(u(r))d\theta_{-s}\theta_{T_{i}(\omega)}\omega\bigg|\\
    &\le &  c|u_0|+c(1+\|u\|_{\beta, \sim, 0,T(\theta_{T_i(\omega)}\omega)})|||\theta_{-s}\theta_{T_{i}(\omega)}\omega|||_{\beta^\prime,0,s},
\end{eqnarray*}
and because $|||\theta_{-s}\theta_{T_{i}(\omega)}\omega|||_{\beta^\prime,0,s}\leq |||\theta_{T_{i}(\omega)}\omega|||_{\beta^\prime,\hat T(\theta_{T_i(\omega)}\omega),0}\le \mu$, we obtain that
\begin{equation}\label{eq211}
|\phi(s,\theta_{-s}\theta_{T_{i}(\omega)}\omega,u_0)|\le c |u_0|+c\mu \bigg(1+\frac{c}{1-c \mu} |u_0|+\frac{c\mu}{1-c\mu}\bigg).
\end{equation}
Hence the norm of any element in $E(i,\omega)$ is bounded by
\[
\sup_{u_0\in G(i,\omega)}c|u_0|\bigg(1+\frac{c\mu}{1-c \mu}\bigg)+c\mu \bigg(1+\frac{c\mu}{1-c\mu}\bigg)
\]
for $G(i,\omega)$ introduced in (\ref{eq40b}), which gives the desired property.\\
The second property follows similarly.\qquad
\end{proof}\\

\begin{lemma}\label{t5}
Under the conditions of Theorem \ref{t4}, define the family of sets $\aA$ by $\aA(\omega):=\aA(0,\omega)$, where $\aA(i,\omega)$ has been defined in Theorem \ref{t4}. Then $\aA$ is invariant and attracting for elements from $\dD_{R,V}^0$ with respect to $\phi$.
\end{lemma}

\begin{proof}
For $t>0$ there exists a unique $i^\ast=i^\ast(\omega)\in Z^-$ such that $-t\in ( T_{i^\ast(\omega)-1}(\omega), T_{i^\ast(\omega)}(\omega)]$. Because of $T_{-i^\ast}(\theta_{ T_{i^\ast}}\omega)=- T_{i^\ast}(\omega)$, due also to the cocycle property and the relationship between $\phi$ and $\Phi$, we can conclude that for $D\in \dD_{R,V}^0$
\begin{eqnarray*}
  & & \phi(t,\theta_{-t}\omega,D(\theta_{-t}\omega))\\
  &= &\phi(T_{-i^\ast}(\theta_{ T_{i^\ast}}\omega),\theta_{T_{i^\ast}(\omega)}\omega,
    \phi(t-T_{-i^\ast}(\theta_{ T_{i^\ast}}\omega),\theta_{-t- T_{i^\ast}}\theta_{ T_{i^\ast}}\omega,D(\theta_{-t- T_{i^\ast}}\theta_{T_{i^\ast}}\omega)))\\
    &=&\Phi(-i^\ast,i^\ast, \omega,
   \phi(t-T_{-i^\ast}(\theta_{ T_{i^\ast}}\omega),\theta_{-t- T_{i^\ast}}\theta_{ T_{i^\ast}}\omega,D(\theta_{-t- T_{i^\ast}}\theta_{ T_{i^\ast}}\omega)))\\
    &\subset & \Phi(-i^\ast,i^\ast, \omega,E_{D}(i^\ast,\omega))
\end{eqnarray*}
where $E_D\in \dD_{Z,V}^0$ is defined in Lemma \ref{l15} and that can be also written as
\[
E_D(i,\omega)=\overline{\bigcup_{\tau\in [T_{i-1}(\omega),T_{i}(\omega)]}\bigcup_{u_0\in D(\theta_{\tau}\omega)}
    \{\phi(T_i(\omega)-\tau,\theta_{\tau-T_i(\omega)}\theta_{T_i(\omega)}\omega,u_0)\}}.
\]
Hence $\aA$ attracts $D$ with respect to $\phi$ since
\[
    0 \le{\rm dist}(\phi(t,\theta_{-t}\omega,D(\theta_{-t}\omega)),\aA(\omega))\le {\rm dist}(\Phi(-i,i,\omega,E_D(i,\omega)),\aA(0,\omega)),
\]
and
\[
    \lim_{i\to-\infty}{\rm dist}(\Phi(-i,i,\omega,E_D(i,\omega)),\aA(0,\omega))=0.
\]
Now we prove the invariance of $\aA$.
For $t>0$ we denote by $i^\prime(\omega)$ the largest $i$ such that $T_{i^\prime}(\omega)< t$
for any $\omega\in\Omega$.  Hence by (\ref{eq33}) for $t_2=t,\,t_1=T_{i^\prime}(\omega)$
\begin{eqnarray*}
  \cdots &\le& T_{i^\prime-1}(\omega) <\ T_{-1}(\theta_t\omega)+t\le T_{i^\prime}(\omega)<  t \\
  &\le& T_{i^\prime+1}(\omega)< T_1(\theta_t\omega)+t\le T_{i^\prime+2}(\omega)< T_2(\theta_t\omega)+t\le\cdots.
\end{eqnarray*}
Since $(\aA(i^\prime+j,\omega))_{j\in Z}$ is pullback attracting with respect to $\Phi$, by (\ref{eq32}) the same property is true for the following compact set
\[
    \tilde \aA(j,\theta_t\omega):=\phi(t-T_{i^\prime+j}(\omega)+T_j(\theta_t\omega),\theta_{T_{i^\prime+j}(\omega)}\omega,\aA(i^\prime+j,\omega)).
\]
Because of $|t-T_{i^\prime+j}(\omega)+T_j(\theta_t\omega)|\le 1$ and $\{\aA(j,\omega)\}_{j\in Z}\in\dD_{Z,V}^\nu$, applying the second part of Lemma \ref{l15} we have that $\tilde \aA(j,\theta_t\omega)\in \dD_{Z,V}^\nu$.
Furthermore
\begin{eqnarray*}
   & & \Phi(1,j,\theta_t\omega,\tilde A(j,\theta_t\omega))\\
   &=&\phi(T(\theta_{T_j(\theta_t\omega)}\theta_t\omega),\theta_{T_j(\theta_t\omega)}\theta_t\omega,
    \phi(t+T_j(\theta_t\omega)-T_{i^\prime+j}(\omega),\theta_{T_{i^\prime+j}(\omega)}\omega,\aA(i^\prime+j,\omega)))\\
    &=&\phi(t+T_{j+1}(\theta_t\omega)-T_{i^\prime+j}(\omega),\theta_{T_{i^\prime+j}(\omega)}\omega,\aA(i^\prime+j,\omega))\\
    &=&\phi(t+T_{j+1}(\theta_t\omega)-T_{i^\prime+j+1}(\omega),\theta_{T_{i^\prime+j+1}(\omega)}\omega,\aA(i^\prime+j+1,\omega))=\tilde \aA(j+1,\theta_t\omega)
\end{eqnarray*}
which shows that $\tilde \aA(\cdot,\theta_t\omega)$ is a compact {\em invariant} set in $\dD^\nu_{Z,V}$ for the cocycle $\Phi(\cdot,\cdot,\theta_t\omega,\cdot)$. Since $\tilde \aA(\cdot,\theta_t\omega)\in \dD^\nu_{Z,V}$ it is attracted by $\aA(\cdot,\theta_t\omega)$, and due to the invariance of the former
\begin{eqnarray}\label{a1}
\tilde \aA(j,\theta_t\omega)\subset \aA(j,\theta_t\omega).
\end{eqnarray}
Now we define
\[
    \hat \aA(i^\prime+j+1,\omega):=\phi(T_{i^\prime+j+1}(\omega)-t-T_j(\theta_t\omega),\theta_{T_j(\theta_t\omega)}\theta_t\omega,\aA(j,\theta_t\omega)),
\]
which implies that $\hat \aA$ has qualitatively the same properties as $\tilde \aA$. In particular, $\hat \aA(j,\omega)\subset \aA(j,\omega)$. But we have
\begin{eqnarray}\label{a2}
    && \aA(\theta_t\omega)=\aA(0,\theta_t\omega)\nonumber\\
    &=&\Phi(1,-1,\theta_t\omega,\aA(-1,\theta_t\omega))= \phi(t-T_{i^\prime}(\omega),\theta_{T_{i^\prime}(\omega)}\omega,\hat \aA(i^\prime(\omega),\omega))\\
     &\subset &\phi(t-T_{i^\prime}(\omega),\theta_{T_{i^\prime}(\omega)}\omega, \aA(i^\prime(\omega),\omega))=\tilde \aA(0,\theta_t\omega)\nonumber
\end{eqnarray}
such that, by (\ref{a1}) and (\ref{a2}), $\tilde \aA(0,\theta_t\omega)=\aA(\theta_t\omega)$ and
\[
    \phi(t,\omega,\aA(\omega))=\phi(t-T_{i^\prime}(\omega),\theta_{T_{i^\prime}(\omega)}\omega,\Phi(i^\prime,0,\omega,\aA(0,\omega))=\aA(\theta_t\omega),
\]
hence $\aA$ is invariant. \qquad \end{proof}

\vskip.5cm

We would like to emphasize the following: the set ${\cal A}$ is {\em only} called attracting because it is contained in  ${\cal D}_{R,V}^{\frac{\nu}{\check d}}$, with $\check d$ given by (\ref{dd}) below, but we cannot prove that $\aA$ attracts the sets of this family. Let us explain these statements with more details. Consider
\begin{equation}\label{dd}
1\ge \hat d:=\limsup_{i\to-\infty}\frac{|T_i(\omega)|}{|i|}>\liminf_{i\to-\infty}\frac{|T_i(\omega)|}{|i|}=:{\check d.}
\end{equation}
First ${\cal A}\in {\cal D}_{R,V}^{\frac{\nu}{\check d}}$, since as $\aA(\theta_{T_i(\omega)}\omega) =\aA(i,\omega)$, we obtain
\begin{eqnarray*}
   &&\limsup_{i\to-\infty}\frac{\sup_{u\in \aA(\theta_{T_i(\omega)}\omega)}|u|}{|T_i(\omega)|}\le
\limsup_{i\to-\infty}\frac{\sup_{u\in \aA(i,\omega)}|u|}{|i|}\limsup_{i\to-\infty}\frac{|i|}{|T_i(\omega)|} \le \frac{\nu}{\check d}
\end{eqnarray*}
and then the property follows.
Secondly, we can assume that there exists a set $D\in {\cal D}_{R,V}^{\frac{\nu}{\check d}}$ and an small $\eps>0$ such that 
\[
\limsup_{i\to-\infty}\frac{\sup_{u\in D(\theta_{T_i(\omega)}\omega)}|u|}{|i|}=
\limsup_{i\to-\infty}\frac{\sup_{u\in D(\theta_{T_i(\omega)}\omega)}|u|}{|T_i(\omega)|}\limsup_{i\to-\infty}\frac{|T_i(\omega)|}{|i|}=\frac{(\nu-\eps) \hat d}{\check d}
\]
which is equal to $(\nu-\eps)$ if and only $\hat d=\check d$. Therefore, if we do not assume this last equality we get a contradiction (since the pullback attractor of $\Phi$ attracts sets $D\in {\cal D}_{R,V}^{\nu}$) and hence we cannot claim that $\aA$ attracts the sets of ${\cal D}_{R,V}^{\frac{\nu}{\check d}}$.

\section{Attractors for random dynamical systems of SPDEs with fractional Brownian motion}

In this section we study the non-autonomous dynamical system under measurability assumptions. In particular, we need to introduce a metric dynamical system. It will be crucial that the  integrals with H{\"o}lder continuous integrators are  defined {\em pathwise}. This is a qualitative difference
to the definition of the classical stochastic integral where the integrand is a white noise. We recall that the pathwise definition of the former integral just gave us the non--autonomous dynamical system $\varphi$ from Section \ref{s3}.

A one dimensional fBm is a centered Gau{\ss}-process on $R$ with {\em Hurst parameter} $H\in(0,1)$ having the covariance

\[
    R(s,t)=\frac12(|t|^{2H}+|s|^{2H}-|t-s|^{2H})
\]
defined on an appropriate probability space.
Similarly we can define an fBm with values in $V$ and covariance $Q$, where $Q$ is a positive symmetric operator on $V$ of trace class.
It is known that $Q$ has a discrete spectrum $(q_i)_{i\in N}$ related to the complete orthonormal system
in $V$ given by $(f_i)_{i\in N}$.  This process has a version $\omega$ in $C_0(R,V)$, the space of continuous functions which are zero at zero, and $q_i^{-\frac12}\pi_i\omega$ are iid one dimensional fBm where $\pi_i$ is the
projection on the $i$-th mode of the base $(f_i)_{i\in N}$, such that we can define the integrals as in Subsection \ref{ss2.2}.
For simplicity we identify in the following the base $(f_i)_{i\in N}$ with the base $(e_i)_{i\in N}$ .

In what follows we consider a canonical version of this process given by the probability space $(C_0(R,V),\bB(C_0(R,V)),P)$ where $P$ is the Gau{\ss}-measure generated by the fBm.
On this probability space we can also introduce the  shift $\theta_t\omega(\cdot)=\omega(\cdot+t)-\omega(t)$.\\

\begin{lemma}\label{l4}
$(C_0(R,V),\bB(C_0(R,V)),P,\theta)$ is an ergodic metric dynamical system.
\end{lemma}

The proof of this lemma can be found in \cite{MasSchm04} and with a deeper analysis in \cite{GS11}.\\

This (canonical) process has a version which is ${\beta^{\prime\prime}}$-H{\"o}lder continuous on any interval $[-k,k]$  for $\beta^{\prime\prime}<H$.

Let us denote by $\Omega_{\beta^{\prime\prime}}\subset C_0(R,V)$
the set of elements which are $\beta^{\prime\prime}$-H{\"o}lder continuous on any interval $[-k,k],\,k\in N$, and are zero at $0$.

\begin{lemma} (see \cite{ChGGSch12})
We have $\Omega_{\beta^{\prime\prime}}\in\bB(C_0(R,V))$. In addition, $\Omega_{\beta^{\prime\prime}}$ is $(\theta_t)_{t\in R}$-invariant.
\end{lemma}

Choose $\mu,\,\nu$ such that (\ref{eq21}) holds. Then by the ergodic theory there exists
a $(\theta_t)_{t\in R}$--invariant set $\Omega_1\in\bB(C_0(R,V))$ of full measure such that

\begin{eqnarray}
  &&  \lim_{t\to\pm\infty}\frac{1}{t}\int_0^t \bigg(\frac{\sup_{r\in[0,1]}|||\theta_{r+q}\omega|||_{\beta^{\prime\prime},-1,0}+\mu}{\mu}\bigg)^{\frac{1}{\beta^{\prime\prime}-\beta^\prime}}dq \nonumber \\[-1.5ex]
    \label{eq35}\\[-1.5ex]
&=& E_P \bigg(\frac{\sup_{r\in[0,1]}|||\theta_r\omega|||_{\beta^{\prime\prime},-1,0}+\mu}{\mu}\bigg)^{\frac{1}{\beta^{\prime\prime}-\beta^\prime}}=:d^{-1}<\infty.\nonumber
\end{eqnarray}
For the construction of the invariant set $\Omega_1$ we refer to Arnold \cite{Arn98} Page 538 {\em Ergodic Theorem (iv)}.
The above expectation can be controlled by choosing a $Q$ with small trace such that $d$ is close to 1, see for instance the construction of the random variable $K$ in Kunita \cite{Kunita90} Theorem 1.4.1.
Let us also choose ${\rm tr}\,Q$ small enough such that (\ref{eq39}) holds true, i.e.
\[
    \nu+d>1.
\]
 In addition, there exists a $(\theta_t)_{t\in R}$-invariant set $\Omega_2\in\bB(C_0(R,V))$ of full measure such that $R\ni t\mapsto\sup_{r\in[0,1]}|||\theta_{t+r}\omega|||_{\beta^{\prime\prime},-1,0}$ has a sublinear growth, which follows by  Arnold \cite{Arn98} Proposition 4.1.3.
Indeed, $\sup_{r\in[0,1]}|||\theta_r\omega|||_{\beta^{\prime\prime},-1,0}\le |||\omega|||_{\beta^{\prime\prime},-1,1}$ where the right hand side  of this inequality has finite moments of any order, by Kunita \cite{Kunita90} Theorem 1.4.1.

Setting $\Omega=\Omega_1\cap\Omega_2\cap\Omega_{\beta^{\prime\prime}}$, let $\fF$ be the trace-$\sigma$-algebra of $\bB(C_0(R,V))$ with respect to $\Omega$, and let us consider
the restriction of the measure $P$ to this $\sigma$-algebra, which we denote again by $P$. Since
$\Omega$ is $(\theta_t)_{t\in R}$-invariant it is not hard to see that the restriction of $\theta$ to $R\times\Omega$ is a random flow, and that $(\Omega,\fF,P,\theta)$ is an ergodic metric dynamical system, see \cite{CGSV10} for a general proof of these properties.

We are now in a position to derive from the non--autonomous dynamical system $\phi$ a random dynamical system.

\begin{lemma}\label{l17}
(i) Let $\phi$ be the solution map to (\ref{eq7}). Then this mapping is $\bB(R^+)\otimes\fF\otimes \bB(V),\bB(V)$ measurable.

(ii) The stopping times $T,\, \hat T$ are measurable.
\end{lemma}

\begin{proof}
First note that  $(r,\omega)\mapsto \omega(r)$ is measurable such that $(r,\omega)\mapsto D_{T-}^{1-\alpha}\omega_{T-}[r]$ is measurable.
Assuming in addition that $(r,\omega)\mapsto k(r,\omega)\in L_2(V)$ is measurable and that for any $\omega$ the product $\|D_{0+}^\alpha k(\cdot,\omega)[r]\|_{L_2(V)}|D_{T-}^{1-\alpha}\omega_{T-}[r]|$ is integrable, then the stochastic integral

\[
    \omega\mapsto\int_0^TS(T-r)G(u(r,\omega))d\omega
\]
is measurable, since $k(r,\omega)=G(u(r,\omega))$ has the above properties.

Since the solution $u$ of (\ref{eq7}) is given by the Banach fixed point theorem, it can be constructed by successive iterations in the space $C^{\beta,\sim}([0,T];V)$, starting this procedure with the constant function $u_0$. Although these approximations are given on different time intervals $[T_i(\omega),T_{i+1}(\omega)]$ they can be finitely concatenated to one approximation on $[0,T]$ converging to a solution of (\ref{eq7}) in the separable Banach space $C([0,T];V)$ which is measurable.

If we pick initial conditions in a sufficiently small neighborhood of $u_0$, the contraction constants for the corresponding mapping in the Banach fixed point theorem can be chosen to be the same, and therefore $u_0\mapsto \phi(t,\omega,u_0)$ is continuous. Then, by Castaing and Valadier \cite{CasVal77}, the mapping $\phi$ is measurable on $\bB([0,T])\otimes \fF\otimes \bB(V),\bB(V)$. Hence, the
$\bB(R^+)\otimes \fF\otimes \bB(V),\bB(V)$ measurability of this expression follows immediately.

(ii) Since $\omega\mapsto |||\omega|||_{\beta^\prime,0,\tau}$ is measurable, the measurability of $\omega\mapsto T(\omega)$ follows. In the same manner we can argue for $\hat T$.
\qquad \end{proof}\\

In order to establish the growth of the stopping times we need the following lemma.
\begin{lemma}\label{N}
Let $N(\omega)\in N$ be the random number of stopping times in $[-1,0]$ defined by (\ref{eq29}) and (\ref{stop}).
Then we have for $\omega\in\Omega_{\beta^{\prime\prime}}$
\[
    N(\omega)\le \bigg(\frac{|||\omega|||_{\beta^{\prime\prime},-1,0}+\mu}{\mu}\bigg)^{\frac{1}{\beta^{\prime\prime}-\beta^\prime}}.
\]
\end{lemma}

\begin{proof}
The following estimates follow easily on $\Omega_{\beta^{\prime\prime}}$:
\begin{eqnarray*}
    \mu&=&\sup_{\hat T(\omega)\le s<t\le 0}\frac{|\omega(t)-\omega(s)|}{|t-s|^{\beta^\prime}}+\mu(-\hat T(\omega))^{1-\beta^\prime}\\
    &\le& \bigg(\sup_{\hat T(\omega)\le s<t\le 0}\frac{|\omega(t)-\omega(s)|}{|t-s|^{\beta^{\prime\prime}}}+\mu(-\hat T(\omega))^{1-\beta^{\prime \prime}}\bigg)(-\hat T(\omega))^{\beta^{\prime\prime}-\beta^\prime}\\
    &\le &(|||\omega|||_{\beta^{\prime\prime},-1,0}+\mu)(-\hat T(\omega))^{\beta^{\prime\prime}-\beta^\prime}
\end{eqnarray*}
such that
\[
    |\hat T(\omega)|\ge \bigg(\frac{\mu}{|||\omega|||_{\beta^{\prime\prime},-1,0}+\mu}  \bigg)^{\frac{1}{\beta^{\prime \prime}-\beta^\prime}}.
\]
The same estimate holds for $\hat T(\theta_{T_{-i}(\omega)}\omega)$ as long as {\bf  $\hat T(\theta_{T_{-i}(\omega)}\omega)+T_{-i}(\omega)=T_{-i-1}(\omega)\ge-1$.}
For the smallest number $i$ having this property we set $N(\omega)=i+1$.
Then due to the definition of the stopping times (\ref{stop}) we conclude that
$$
    1\ge |T_{-N(\omega)}(\omega)|=\sum_{i=0}^{N(\omega)-1}|\hat T(\theta_{T_{-i}(\omega)}\omega)|\ge N(\omega)\bigg(\frac{\mu}{|||\omega|||_{\beta^{\prime\prime},-1,0}+\mu}\bigg)^{\frac{1}{\beta^{\prime \prime}-\beta^\prime}},
    $$
and then the result follows. \qquad\end{proof}\\

\begin{lemma}\label{growth}
On $\Omega$ we have

\[
    \liminf_{k\to-\infty}\frac{T_k(\omega)}{k}\ge d>0.
\]
On this set $|T(\theta_{T_i(\omega)}\omega)|^{-\beta}$ is subexponentially growing for $i\to-\infty$.
\end{lemma}

\begin{proof}
Let $M_j(\omega)$ be the number of stopping times $T_k(\omega)$ in $(j-1,j]$.

For $j\in- Z$ let $k_j(\omega)$ be the biggest integer such that $T_{k_j(\omega)} \leq j$. Then choosing $t_2=j$ and $t_1=T_{k_j(\omega)}(\omega)$ from (\ref{eq33}) we know
\begin{eqnarray*}
    \cdots & < &T_{k_j(\omega)-2}(\omega)        \le j+ \hat T(\theta_j\omega)+\hat T(\theta_{\hat T(\theta_j\omega)+j}\omega) =j+T_{-2}(\theta_j\omega)\\&< & T_{k_j(\omega)-1}(\omega) \le j+\hat T(\theta_j\omega)
     =j+ T_{-1}(\theta_j\omega)  <   T_{k_j(\omega)}(\omega)   \le j
\end{eqnarray*}
such that $M_j(\omega)=N(\theta_j\omega)$ or $M_j(\omega)=N(\theta_j\omega)-1$, where $N(\omega)$ has been introduced in Lemma \ref{N}. Hence

\begin{eqnarray*}
  \limsup_{n\to\infty,n>1}\frac{\sum_{j=0}^{n-1}M_{-j}(\omega)}{n-1} &\le &\limsup_{n\to\infty,n>1}\frac{\sum_{j=0}^{n-1}N(\theta_{-j}\omega)}{n-1}\\
  &\le & \limsup_{n\to\infty,n>1}\frac{\sum_{j=0}^{n-1} \bigg(\frac{|||\theta_{-j}\omega|||_{\beta^{\prime\prime},-1,0}+\mu}{\mu}\bigg)^{\frac{1}{\beta^{\prime\prime}-\beta^\prime}}}{n-1}
\end{eqnarray*}
which follows from Lemma \ref{N} for $\omega\in \Omega_{\beta^{\prime\prime}}$. The terms under the sum of the numerator on the right hand side can be estimated by
\[
    \bigg(\frac{\sup_{r\in[0,1]}|||\theta_{r+q-j}\omega|||_{\beta^{\prime\prime},-1,0}+\mu}{\mu}\bigg)^{\frac{1}{\beta^{\prime\prime}-\beta^\prime}}\;{\rm for any}\ q\in[-1,0]
\]
 and hence
\[
     N(\theta_{-j}\omega)\le\int_{-1}^0 \bigg(\frac{\sup_{r\in[0,1]}|||\theta_{r+q-j}\omega|||_{\beta^{\prime\prime},-1,0}+\mu}{\mu}\bigg)^{\frac{1}{\beta^{\prime\prime}-\beta^\prime}}dq.
\]
Then we obtain the estimate for $n>1$
\[
  \frac{\sum_{j=0}^{n-1}N(\theta_{-j}\omega)}{n-1} \le \frac{1}{n-1}\int_{-n}^0 \bigg(\frac{\sup_{r\in[0,1]}|||\theta_{r+q}\omega|||_{\beta^{\prime\prime},-1,0}+\mu}{\mu}\bigg)^{\frac{1}{\beta^{\prime\prime}-\beta^\prime}}dq
\]
where the right hand side converges to $1/d$ for $\omega\in\Omega_1$ and $n\to\infty$ by (\ref{eq35}).
Then we have
\[
    \liminf_{n \to\infty}\frac{\bigg|T_{-\sum_{j=0}^{n-1} M_{-j}(\omega)}(\omega)\bigg|}{\sum_{j=0}^{n-1} M_{-j}(\omega)}\ge
    \liminf_{n \to\infty}\frac{(n-1)}{\sum_{j=0}^{n-1} M_{j-1}(\omega)}\geq d.
\]
To see that finally the conclusion holds true, choose for every $k\in -N$ an $n=n(k,\omega)$ such that
\[
    -\sum_{j=0}^{n}M_{-j}(\omega)< k\le -\sum_{j=0}^{n-1}M_{-j}(\omega).
\]
Thus
\begin{eqnarray*}
   \liminf_{k\to -\infty}\frac{|T_{-k}(\omega)|}{|k|}&\ge& \liminf_{n \to-\infty}\frac{\bigg|T_{-\sum_{j=0}^{n-1} M_{-j}(\omega)}(\omega)\bigg|}{\sum_{j=0}^{n} M_{-j}(\omega)}
   \ge \liminf_{n \to-\infty}\frac{\bigg|T_{-\sum_{j=0}^{n-1} M_{-j}(\omega)}(\omega)\bigg|}{\sum_{j=0}^{n-1} M_{-j}(\omega)+N(\theta_{-n}\omega)}\\
   &= &\liminf_{n \to-\infty}\frac{n-1}{\sum_{j=0}^{n-1} M_{-j}(\omega)}\ge d
\end{eqnarray*}
by the sublinear convergence of $n\mapsto N(\theta_n\omega)$ (thanks to the ergodic theorem) and the at least asymptotical linear growth of $n\mapsto \sum_{j=0}^{n-1} M_{-j}(\omega)+N(\theta_{-n}\omega)$.

To prove the second part of the statement, note that similar to the estimate of $\hat T$ in the proof of Lemma \ref{N}, for $\omega\in\Omega$,
\[
|T(\theta_{T_i(\omega)}\omega)|^{-\beta}\le
\bigg(\frac{|||\theta_{T_i(\omega)}\omega|||_{\beta^{\prime\prime},0,1}+\mu}{\mu}\bigg)^{\frac{\beta}{\beta^{\prime\prime}-\beta^\prime}}\\
\]
for large $|i|$, which follows from the fact that the mapping $i\mapsto |||\theta_{T_i(\omega)}\omega|||_{\beta^{\prime\prime},0,1}$ grows sublinearly in $\Omega \subset \Omega_2$, and therefore, the previous inequality shows the subpolynomial growth of $|T(\theta_{T_i(\omega)}\omega)|^{-\beta}$. \qquad \end{proof}\\

We can finally prove the main result of this section:\\

\begin{theorem}
The pullback attractor stated in Theorem \ref{t5} is a random attractor attracting the random tempered sets $\hat{\dD}$.

\end{theorem}
\begin{proof}
Since the stopping times $T_i$ are random variables we obtain that the radii of the absorbing balls $R(i,\omega)$ are random variables, and therefore the balls
$B(i,\omega)$ are random sets. Note that
\[
  \aA(\omega)=\bigcap_{j\in Z^+}\overline{\bigcup_{i\in Z^+,i\ge j}\overline{\Phi(i,-i,\omega,B(-i,\omega))}}.
\]
$\Phi$ inherits the measurable properties of $\phi$ and in particular $u\mapsto \Phi(i,-i,\omega,u)$ is continuous. Hence $\overline{\Phi(i,-i,\omega,B(-i,\omega))}$ is a random set. Then by \cite{FlaSchm96} we have that $\aA$ is a random set.
On the other hand, we have that $\aA(\theta_{T_i(\omega)}\omega)\subset B(i,\omega)$ and $\aA\in \dD_{R,V}^{\frac{\nu}{\check d}}$, for $\check d$ given by (\ref{dd}). Furthermore, Lemma \ref{l8} shows that
$\aA$ is a set in $\hat\dD$, and since $\hat \dD\subset \dD_{R,V}^0$ the conclusion follows from Lemma \ref{t5}.
\qquad\end{proof}

\section{Appendix}

We present here the proof of Lemma \ref{l9-u}.\\

\begin{proof}We have to estimate
\begin{eqnarray*}
 \|u\|_{\beta,T_{i-1}(\theta_{T_j}\omega),T_{i}(\theta_{T_j}\omega)} =\sup_{T_{i-1}(\theta_{T_j}\omega)\le t\le T_{i}(\theta_{T_j}\omega)}|u(t)|+|||u|||_{\beta,T_{i-1}(\theta_{T_j}\omega),T_{i}(\theta_{T_j}\omega)}.
 \end{eqnarray*}
First of all, considering $u_0\in V_\beta$, from (\ref{eq1}) and (\ref{eq2})  we obtain
\[\|S(\cdot)u_0\|_{\beta,T_{i-1}(\theta_{T_j}\omega),T_{i}(\theta_{T_j}\omega)}\le c e^{-\lambda_1 T_{i-1}(\theta_{T_j}\omega)}|u_0|_{V_\beta}.\]
Thanks to additivity of the above integral, see \cite{ChGGSch12}, for $t\in [T_{i-1}(\theta_{T_j} \omega), T_{i}(\theta_{T_j} \omega)]$ the following splitting holds true:
\begin{eqnarray*}
\int_0^t S(t-r)G(u(r))d\theta_{T_j}\omega&=&\int_{T_{i-1}(\theta_{T_j} \omega)}^t S(t-r)G(u(r))d\theta_{T_j}\omega\\
&+&\sum_{m=1}^{i-1} \int_{T_{m-1}(\theta_{T_j} \omega)}^{T_{m}(\theta_{T_j} \omega)} S(t-r)G(u(r))d\theta_{T_j}\omega,
\end{eqnarray*}
and therefore, in particular,
\begin{eqnarray}
 &  & \bigg| \bigg| \bigg|\int_0^\cdot S(\cdot-r)G(u(r))d\theta_{T_j}\omega\bigg| \bigg| \bigg|_{\beta,T_{i-1}(\theta_{T_j}\omega),T_{i}(\theta_{T_j}\omega)} \nonumber \\
  &  \le & \bigg| \bigg| \bigg|\int_{T_{i-1}(\theta_{T_j}\omega)}^\cdot S(\cdot-r)G(u(r))d\theta_{T_j}\omega\bigg| \bigg| \bigg|_{\beta,T_{i-1}(\theta_{T_j}\omega),T_{i}(\theta_{T_j}\omega)} \label{eq15}\\
    &+&\sum_{m=1}^{i-1}\bigg| \bigg| \bigg|\int_{T_{m-1}(\theta_{T_j}\omega)}^{T_{m}(\theta_{T_j}\omega)} S(\cdot-r)G(u(r))d\theta_{T_j}\omega\bigg| \bigg| \bigg|_{\beta,T_{i-1}(\theta_{T_j}\omega),T_{i}(\theta_{T_j}\omega)}.\nonumber
\end{eqnarray}
Regarding the first term on the right hand side of (\ref{eq15}), notice that for $s<t\in [T_{i-1}(\theta_{T_j} \omega), T_{i}(\theta_{T_j} \omega)]$, we have that

\begin{eqnarray}
 \qquad  & \bigg|&\int_{T_{i-1}(\theta_{T_j}\omega)}^t S(t-r)G(u(r))d\theta_{T_j}\omega-\int_{T_{i-1}(\theta_{T_j}\omega)}^s S(s-r)G(u(r))d\theta_{T_j}\omega\bigg| \nonumber \\[-1.5ex]
    &=&\bigg|\int_{s}^t S(t-r)G(u(r))d\theta_{T_j}\omega+\int_{T_{i-1}(\theta_{T_j}\omega)}^s (S(t-s)-{\rm Id})S(s-r)G(u(r))d\theta_{T_j}\omega\bigg| \nonumber \\[-1.5ex]
    \label{p1}\\[-1.5ex]
    &\le& \bigg|\int_{s^\prime}^{t^\prime} S(t^\prime-r)G(u(r+T_{i-1}(\theta_{T_j}\omega)))d\theta_{T_{i+j-1}}\omega\bigg| \nonumber \\[-1.5ex]
        & +&c(t^\prime-s^\prime)^\beta
    \bigg|\int_0^{s^\prime} S(s^\prime-r)G(u(r+T_{i-1}(\theta_{T_j}\omega)))d\theta_{T_{i+j-1}}\omega\bigg|_{V_\beta}\nonumber.
\end{eqnarray}
where we have performed the change of variable $\tilde r=r-T_{i-1}(\theta_{T_j}\omega)$ and later on renamed $\tilde r$ as $r$, see Lemma \ref{l3}. Therefore, in the previous expression, $s^\prime=s-T_{i-1}(\theta_{T_j}\omega)$, $t^\prime=t-T_{i-1}(\theta_{T_j}\omega)$. Note that then $s^\prime<t^\prime \in [0, T(\theta_{T_{i+j-1}} \omega)]$.

The first term on the right hand side of $(\ref{p1}) $ can be estimated in the following way:
\begin{eqnarray*}
& \bigg| & \int_{s^\prime}^{t^\prime} S(t^\prime-r)G(u(r+T_{i-1}(\theta_{T_j}\omega)))d\theta_{T_{i+j-1}}\omega\bigg| \\
 & \leq&  c|||\theta_{T_{i+j-1}}\omega|||_{\beta^\prime,0,T(\theta_{T_{i+j-1}}\omega)} \int_{s^\prime}^{t^\prime}
    \bigg(\frac{(c_G+c_{DG}|u(r+T_{i-1}(\theta_j \omega))|)e^{-\lambda_1(t^\prime-r)}}{(r-s^\prime)^\alpha}\\
     &\qquad& +\int_{s^\prime}^r\frac{(c_G+c_{DG}|u(r+T_{i-1}(\theta_j \omega))|)e^{-\lambda_1(t^\prime-r)}(r-q)^{\beta}}
    {(t^\prime-r)^{\beta}(r-q)^{\alpha+1}}dq\\
    &\qquad &+ \int_{s^\prime}^r\frac{c_{DG}e^{-\lambda_1(t^\prime-r)}|u(r+T_{i-1}(\theta_j \omega))-u(q+T_{i-1}(\theta_j \omega))|}{(r-q)^{\alpha+1}}dq\bigg)
   (t^\prime-r)^{\beta^\prime+\alpha-1}dr\\
   &\leq &  c|||\theta_{T_{i+j-1}}\omega|||_{\beta^\prime,0,T(\theta_{T_{i+j-1}}\omega)} (1+\|u\|_{\beta, T_{i-1}(\theta_{T_j}\omega),
    T_{i}(\theta_{T_j}\omega)}) (t^\prime -s^\prime)^{\beta^\prime}\\
    &\leq & c\mu(1+\|u\|_{\beta,T_{i-1}(\theta_{T_j}\omega),
    T_{i}(\theta_{T_j}\omega)})(t-s)^\beta.
\end{eqnarray*}
We have used the fact that, by definition, in the interval $[0,T(\theta_{T_i+j-1} \omega)]$ the norm of the noise paths are smaller than $\mu$.

Furthermore, for the last term on the right hand side of $(\ref{p1})$, by applying Lemma \ref{l11} yields that
\begin{eqnarray*}
 &  &(t^\prime-s^\prime)^\beta
    \bigg|\int_{0}^{s^\prime} S(s^\prime-r)G(u(r+T_{i-1}(\theta_{T_j}\omega)))d\theta_{T_{i+j-1}}\omega\bigg|_{V_\beta} \\
    &  \le &(t^\prime-s^\prime)^\beta c\mu(1+\|u\|_{\beta,T_{i-1}(\theta_{T_j}\omega),
    T_{i}(\theta_{T_j}\omega)})\\
    & \le & (t-s)^\beta c\mu(1+\|u\|_{\beta,T_{i-1}(\theta_{T_j}\omega),
    T_{i}(\theta_{T_j}\omega)}).
\end{eqnarray*}
Moreover, for the terms of the sum on (\ref{eq15}), for $t\in [T_{i-1} (\theta_{T_j} \omega), T_{i} (\theta_{T_j} \omega)]$,
\begin{eqnarray*}
    & &\int_{T_{m-1}(\theta_{T_j}\omega)}^{T_{m}(\theta_{T_j}\omega)}  S(t-r)G(u(r))d\theta_{T_j}\omega\\
    &=& \int_{0}^{T(\theta_{T_{m+j-1}\omega)}} S(t-(r+T_{m-1}(\theta_{T_j}\omega)))G(u(r+T_{m-1}(\theta_{T_j}\omega)))d\theta_{T_{m+j-1}}\omega\\
    &= & \int_{0}^{T(\theta_{T_{m+j-1}\omega)}} S(t^\prime+T_{i-1}(\theta_{T_j}\omega)-r-T_{m-1}(\theta_{T_j}\omega))G(u(r+T_{m-1}(\theta_{T_j}\omega)))d\theta_{T_{m+j-1}}\omega\\
    &= &S(t^\prime+T_{i-1}(\theta_{T_j}\omega)-T_{m}(\theta_{T_j}\omega))\\
    && \times \int_0^{T(\theta_{T_{m+j-1}})} S(T(\theta_{T_{m+j-1}}\omega)-r)
    G(u(r+T_{m-1}(\theta_{T_j}\omega)))
    d\theta_{T_{m+j-1}}\omega
\end{eqnarray*}
for $t^\prime=t-T_{i-1}(\theta_{T_j}\omega) \in [0, T(\theta_{T_{i+j-1}}\omega)]$. We have used above that $T_{m-1}(\theta_{T_j}\omega) =T_{m}(\theta_{T_j}\omega)-T(\theta_{m+j-1}\omega)$. Therefore, for $s^\prime<t^\prime\in [0, T(\theta_{T_{j+i-1}}\omega)]$ we have

\begin{eqnarray*}
     &&\bigg|S(t^\prime+T_{i-1}(\theta_{T_j}\omega)-T_{m}(\theta_{T_j}\omega))\\
     && \times \int_0^{T(\theta_{T_{m+j-1}})} S(T(\theta_{T_{m+j-1}}\omega)-r)
    G(u(r+T_{m-1}(\theta_{T_j}\omega)))
    d\theta_{T_{m+j-1}}\omega \\
    &-&S(s^\prime+T_{i-1}(\theta_{T_j}\omega)-T_{m}(\theta_{T_j}\omega))\\
    && \times \int_0^{T(\theta_{T_{m+j-1}})} S(T(\theta_{T_{m+j-1}}\omega)-r)
    G(u(r+T_{m-1}(\theta_{T_j}\omega)))
    d\theta_{T_{m+j-1}}\omega\bigg|\\
    &\leq &\bigg| (S(t^\prime-s^\prime)-{\rm Id})S(s^\prime+T_{i-1}(\theta_{T_j}\omega)-T_{m}(\theta_{T_j}\omega)) \\
    &\qquad& \times  \int_0^{T(\theta_{T_{m+j-1}})} S(T(\theta_{T_{m+j-1}}\omega)-r)
    G(u(r+T_{m-1}(\theta_{T_j}\omega)))
    d\theta_{T_{m+j-1}}\omega\bigg|\\
    &\leq& (t^\prime-s^\prime)^\beta e^{-\lambda_1(T_{i-1}(\theta_{T_j}\omega)-T_m(\theta_{T_j}\omega))} \\
    &\qquad & \times \bigg|\int_0^{T(\theta_{T_{m+j-1}})}  S(T(\theta_{T_{m+j-1}}\omega)-r)
    G(u(r+T_{m-1}(\theta_{T_j}\omega)))
    d\theta_{T_{m+j-1}}\omega\bigg|_{V_\beta}\\
    & \leq &(t-s)^\beta e^{-\lambda_1(T_{i-1}(\theta_{T_j}\omega)-T_m(\theta_{T_j}\omega))} c\mu(1+\|u\|_{\beta,T_{m-1}(\theta_{T_j}\omega),
    T_{m}(\theta_{T_j}\omega)}).
\end{eqnarray*}
Note that to estimate
$$\displaystyle{\sup_{T_{i-1}(\theta_{T_j}\omega)\le t\le T_{i}(\theta_{T_j}\omega)}|u(t)|}$$
 it is enough to estimate $ |u(T_{i-1}(\theta_{T_j}\omega))|$ and $|||u|||_{\beta, T_{i-1}(\theta_{T_j}\omega), T_{i}(\theta_{T_j}\omega)}$. For the former, following the same steps as above, we obtain
\begin{eqnarray*}
 &  & |u(T_{i-1}(\theta_{T_j}\omega))|\le |S(T_{i-1}(\theta_{T_j}\omega))u_0|+\bigg|\int_0^{T_{i-1}(\theta_{T_j} \omega)} S(T_{i-1}(\theta_{T_j} \omega)-r)G(u(r))d\theta_{T_j}\omega\bigg|\\
    &\le & c e^{-\lambda_1 T_{i-1}(\theta_{T_j}\omega)}|u_0|+ \sum_{m=1}^{i-1}\bigg|\int_{T_{m-1}(\theta_{T_j}\omega)}^{T_{m}(\theta_{T_j}\omega)} S(T_{i-1}(\theta_{T_j} \omega)-r)G(u(r))d\theta_{T_j}\omega\bigg|\\
    &\leq  &c e^{-\lambda_1 T_{i-1}(\theta_{T_j}\omega)}|u_0|+\sum_{m=1}^{i-1} \bigg|S(T_{i-1}(\theta_{T_j} \omega)-T_{m}(\theta_{T_j}\omega))\\
    &&\qquad  \times \int_0^{T(\theta_{T_{m+j-1}})} S(T(\theta_{T_{m+j-1}}\omega)-r)
    G(u(r+T_{m-1}(\theta_{T_j}\omega)))
    d\theta_{T_{m+j-1}}\omega\bigg|\\
    &\leq & c e^{-\lambda_1 T_{i-1}(\theta_{T_j}\omega)}|u_0|+c\mu \sum_{m=1}^{i-1} e^{-\lambda_1(T_{i-1}(\theta_{T_j} \omega)-T_{m}(\theta_{T_j}\omega))} (1+\|u\|_{\beta,T_{m-1}(\theta_{T_j}\omega),
    T_{m}(\theta_{T_j}\omega)}).
    \end{eqnarray*}
Collecting all the previous estimates we obtain the conclusion, renaming all appearing constants again as $c$. \qquad\end{proof}



\begin{thebibliography}{10}



\bibitem{Arn98}
{\sc L. Arnold}, {\em Random dynamical systems},
Springer Monographs in Mathematics, Springer-Verlag, Berlin, 1998.

\bibitem{BaVis92}
{\sc A.~V. Babin and M.~I. Vishik}, {\em Attractors of evolution equations},
Studies in Mathematics and its Applications, 25, North-Holland Publishing Co., Amsterdam, 1992.

\bibitem{CGSV10}
{\sc T. Caraballo, M.~J. Garrido-Atienza,
B.~Schmalfu{\ss} and J. Valero}, {\em Asymptotic behavior of a stochastic semilinear dissipative functional equation without uniqueness of solutions}, Discrete and continuous dynamical systems, series B, 14(2) (2010), pp.~439--455.

\bibitem{CasVal77}
{\sc C.~Castaing and M.~Valadier}, {\em Convex analysis and measurable multifunctions}, Lectures in Mathematics, Vol. 580, Springer-Verlag, Berlin, 1977.

\bibitem{ChGGSch12}
{\sc Y. Chen, H. Gao, M.~J. Garrido-Atienza and B. Schmalfu\ss}, {\em Pathwise solutions of SPDEs and random dynamical systems},  ArXiv:1305.6903.

\bibitem{chueshov}
{\sc I. Chueshov}, {\em
Monotone random systems theory and applications},
Lecture Notes in Mathematics, Vol. 1779, Springer-Verlag, Berlin, 2002.

\bibitem{FlaSchm96}
{\sc F. Flandoli and B. Schmalfu{\ss}}, {\em Random attractors for
the {$3$}{D} stochastic {N}avier-{S}tokes equation with
multiplicative white noise}, Stochastics Stochastics Rep., 59(1-2) (1996), pp.~21--45.

\bibitem{GLS09} {\sc M.~J. Garrido-Atienza, K. Lu and
B.~Schmalfu\ss}, {\em Random dynamical systems for stochastic partial
differential equations driven by a fractional Brownian motion}, Discrete and continuous dynamical systems, series B, 14(2) (2010), pp.~473--493.


\bibitem{Gar12} \sameauthor, {\em Random dynamical systems for stochastic evolution equations driven by a fractional {B}rownian motion
with Hurst parameter in (1/3,1/2]}, in preparation.

\bibitem{GLS12a}
\sameauthor,  {\em Pathwise solutions of stochastic partial differential equations driven by a fractional {B}rownian motion
with Hurst parameter in (1/3,1/2],} arXiv:1205.6735.

\bibitem{GLS12b}
\sameauthor,  {\em Compensated fractional derivatives and stochastic evolution equations}, {\it Comptes Rendus Math\'ematique},  350(23--24) (2012), pp.~1037--1042.

\bibitem{GMS08}
{\sc M.~J. Garrido-Atienza, B. Maslowski and
B.~Schmalfu\ss}, {\em Random attractors for stochastic equations driven
by a fractional Brownian motion}, International Journal of Bifurcation and Chaos, 20(9) (2010), pp.~1--22.

\bibitem{GS11}
{\sc M.~J. Garrido-Atienza and B. Schmalfu{\ss}}, {\em Ergodicity of the infinite dimensional fractional Brownian motion}, Journal of  Dynamics and Differential Equations, 23 (2011), pp.~671--681.


\bibitem{Hale}
{\sc J.~K. Hale}, {\em Asymptotic behavior of dissipative systems},
Mathematical Surveys and Monographs, 25. American Mathematical Society, Providence, RI, 1988.

\bibitem{Kunita90}
{\sc H.~Kunita}, {\em Stochastic flows and stochastic differential equations},
Cambridge University Press, 1990.


\bibitem{MasNua03}
{\sc B. Maslowski and D. Nualart}, {\em Evolution equations driven by a fractional {B}rownian motion}, J. Funct. Anal., 202(1) (2003), pp.~277--305.


\bibitem{MasSchm04}
{\sc B. Maslowski and B. Schmalfu{\ss}}, {\em Random dynamical systems and stationary solutions of differential
  equations driven by the fractional {B}rownian motion}, Stochastic Anal. Appl.,  22(6) (2004), pp.~1577--1607.

\bibitem{Samko}
{\sc S.~G. Samko, A.~A. Kilbas, O.~I. Marichev}, {\em Fractional integrals and derivatives: theory and applications},
Gordon and Breach Science Publishers, Switzerland and Philadelphia, Pa., USA, 1993.

\bibitem{Schm92}
{\sc B.~Schmalfu{\ss}}, {\em Backward cocycles and attractors of
sto\-chas\-tic dif\-feren\-tial
  equa\-tions}, in International Seminar on Applied Mathematics--Non\-linear Dy\-na\-mics:
  Attractor Approximation and Global Behaviour, V.~Reit\-mann, T.~Riedrich, and N.~Koksch, editors, 1992, pp.~185--192.

\bibitem{Schm99a}
\sameauthor, {\em Attractors for the non-autonomous dynamical systems}, in Proceedings {EQUADIFF99}, K.~Gr{\"o}ger,  B.~Fiedler and J.~Sprekels, editors, World Scientific, 2000, pp.~684--690.

\bibitem{Tem97}
{\sc R. Temam}, {\em Infinite-dimensional dynamical systems in mechanics and physics},
Second edition, Applied Mathematical Sciences, 68. Springer-Verlag, New York, 1997.

\bibitem{You36}
{\sc L.C. Young}, {\em An integration of H{\"o}der type, connected with Stieltjes integration}, Acta Math., 67 (1936), pp.~251--282.


\bibitem{Zah98}
{\sc M.~Z{\"a}hle}, {\em Integration with respect to fractal functions and stochastic
  calculus. {I}}, Probab. Theory Related Fields, 111(3)
(1998), pp.~333--374.

\end{thebibliography}
\end{document}